\documentclass{amsart}
\usepackage{amsmath,amsthm,amssymb,amscd}
\usepackage{mathrsfs}
\usepackage{indentfirst}
\usepackage{url}
\usepackage{fullpage}
\usepackage[all]{xy}
\usepackage{bbm}

\newcommand{\X}[1]{\ensuremath{X(#1)}}
\newcommand{\Xw}[2]{\ensuremath{X(#1)_{/#2}}}
\newcommand{\xunrw}[2]{\ensuremath{\left(X(#1),\eta^{(p)}(#1)\right)_{/#2}}}
\newcommand{\ord}[1]{\ensuremath{\eta_p^{\mathrm{ord}}(#1)}}
\newcommand{\et}[1]{\ensuremath{\eta_p^{\mathrm{\acute{e}t}}(#1)}}
\newcommand{\xw}[2]{\ensuremath{\left(X(#1),\eta^{(p)}(#1),\et{#1}\times\ord{#1}\right)_{/#2}}}

\numberwithin{equation}{section}

\newtheorem{main}{Theorem}[section]
\newtheorem{taylor}{Proposition}[section]
\newtheorem*{recipe}{Recipe for a choice of $\lambda$}
\newtheorem{i-adic}{Proposition}[section]
\newtheorem{critical-character}[i-adic]{Proposition}

\begin{document}

\author{Miljan Brako\v cevi\'c}
\title{Two Variable Anticyclotomic $p$-adic L-functions for Hida families}
\date{\today}
\address{Department of Mathematics, McGill University, 
Montreal, QC, Canada H3A 0B9} 
\email{miljan.brakocevic@mcgill.ca}
\thanks{This research work was partially supported by Prof. Haruzo Hida's NSF grant DMS-0753991
through graduate student research fellowship, UCLA Dissertation Year Fellowship, and by McGill University/CICMA grants through postdoctoral research fellowship.}

\begin{abstract}
For the anticyclotomic $p$-adic Rankin--Selberg L-function attached to a fixed Hecke eigenform and an imaginary quadratic field we introduce the second $p$-adic variable by considering Hida families of Hecke eigenforms parametrized by the weight.
\end{abstract}

\maketitle

\section{Introduction}
For a given normalized Hecke newform $f$ of level $\Gamma_0(N)$, $N\geq 1$, weight $k\geq1$, and nebentypus $\psi$, in \cite{Br11a} we constructed a bounded $p$-adic cuspidal measure $\mu_f$ following closely \cite{Ka} where Katz did that for  Eisenstein series. The Mazur--Mellin transform of $\mu_f$ interpolates the ``square roots'' of the central critical Rankin--Selberg L-values, for the Rankin product of the cusp form $f$ and a theta series of arithmetic Hecke characters of an imaginary quadratic field, in the case when the weight of the Hecke character is greater than that of the cusp form. 

The purpose of this paper is to  introduce the second $p$-adic variable by varying $f$ in a Hida family. This two-variable $p$-adic L-function is intimately connected with Howard's classes (\cite{Ho}) in the Galois cohomology of the $p$-adic Galois representation $V_f$ attached to $f$, which were constructed by taking the Kummer images of Heegner points on the modular abelian variety attached to $f$, and then interpolated as $f$ varies in a Hida family. Namely, the two-variable $p$-adic L-function also arises by taking images of the Howard's classes under Ochiai's exponential map interpolating the exponential map
of Bloch--Kato over the self-dual critical twist of a Hida family (\cite{Oc}, see also \cite{Ca}). The two-variable measure also provides a generalization of a result of Datskovsky and Guerzhoy (\cite{DaGe}) on $p$-adic behavior of Taylor coefficients of a Hecke eigenform at a CM point.

To state the main theorem precisely, we first introduce some notation and recall properties of the $p$-adic measure $\mu_f$. Let $M$ be an imaginary quadratic field, $R$ its ring of integers and $p$ a fixed prime so that the following assumption holds throughout the paper:
\[ \tag{ord}  p\; \text{splits into product of primes}\; p=\mathfrak{p}\bar{\mathfrak{p}}\; \text{in}\; R.\]
We fix two embeddings $\iota_\infty :\bar{\mathbb{Q}} \hookrightarrow \mathbb{C}$ and  $\iota_p : \bar{\mathbb{Q}} \hookrightarrow \mathbb{C}_p$ and write $c$ for both complex conjugation of  $\mathbb{C}$ and $\bar{\mathbb{Q}}$ induced by $\iota_\infty$. Then we can choose an embedding $\sigma : M \hookrightarrow \bar{\mathbb{Q}}$ such that the $p$-adic place induced by $\iota_p \circ \sigma$ is distinct from the one induced by $\iota_p \circ c \circ \sigma$. This choice $\Sigma = \{ \sigma \}$ is called $p$-ordinary CM type and its existence is equivalent to (ord). Let $G$ denote the algebraic group  $\mathrm{GL}(2)_{/\mathbb{Q}}$. Let $f$  be a normalized Hecke newform as above and let $\mathbf{f}$ be its corresponding adelic form on  $G(\mathbb{Q})\backslash G(\mathbb{A})$ with central character $\boldsymbol{\psi}:=\psi|\cdot|_{\mathbb{A}}^{-k}$ (see Section 6 of \cite{Br11a}  for definition). All reasonable adelic lifts of $f$ are equal up to twists by a power of the everywhere unramified character $|\mathrm{det}(g)|_{\mathbb{A}}$, and $\mathbf{f}^u(g) := \mathbf{f}|\boldsymbol{\psi}(\mathrm{det}(g))|^{-1/2}$ is a unique one which generates a unitary automorphic representation  $\pi_\mathbf{f}$. We further take the base-change  $\hat{\pi}_{\mathbf{f}}$ to $\mathrm{Res}_{M/\mathbb{Q}}G$. Pick an arithmetic Hecke character $\lambda$ of $M^{\times} \backslash M^{\times}_{\mathbb{A}}$ of $\infty$-type $(k,0)$ such that condition $\lambda|_{\mathbb{A}^\times} = \boldsymbol{\psi}^{-1}$ holds. Write $\lambda^-:=(\lambda \circ c)/|\lambda|$ for the unitary projection. Under this condition, the L-value $L(1/2,\hat{\pi}_\mathbf{f}\otimes \lambda^-)$, regarded as that of the Rankin--Selberg L-function associated to $\mathbf{f}$ and the theta series $\theta(\lambda^-)$ of $\lambda^-$, is critical in the sense of Deligne and central with respect to the functional equation.

Let $N=\prod_ll^{\nu(l)}$ be the prime factorization and denote by $N_{ns}:=\prod_{l \, \text{non-split}}l^{\nu(l)}$ its ``non-split'' part. Consider the order $R_{N_{ns}p^n}:=\mathbb{Z}+N_{ns}p^nR$ and let $\mathrm{Cl}_n^-:=\mathrm{Pic}(R_{N_{ns}p^n})$, that is, the group of $R_{N_{ns}p^n}$-projective fractional ideals modulo globally principal ideals. We define anticyclotomic class group modulo $N_{ns}p^\infty$, $\mathrm{Cl}_\infty^-:=\varprojlim_n\mathrm{Cl}_n^-$ for the projection $\pi_{m+n,n}:\mathrm{Cl}_{m+n}^-\to\mathrm{Cl}_n^-$ taking $\mathfrak{a}$ to $\mathfrak{a}R_{N_{ns}p^n}$. By the class field theory, the group $\mathrm{Cl}_\infty^-$ is isomorphic to the Galois group of the maximal ring class field of conductor $N_{ns}p^\infty$  of $M$. Let $W$ be the ring of Witt vectors with coefficients in the algebraic closure $\bar{\mathbb{F}}_p$ of the finite field of $p$ elements $\mathbb{F}_p$, regarded as a $p$-adically closed discrete valuation ring inside $p$-adic completion $\mathbb{C}_p$ of $\bar{\mathbb{Q}}_p$. We set $\mathcal{W}=\iota_p^{-1}(W)$ which is a strict henselization of $\mathbb{Z}_{(p)}=\mathbb{Q}\cap\mathbb{Z}_p$.

The $W$-valued $p$-adic measure $\mu_f$ in \cite{Br11a} is constructed on $\mathrm{Cl}_\infty^-$ so that its moments interpolate the special values of the Hecke eigen-cusp form $f$ at CM points on a Shimura variety  that underlies the central critical values  $L(\frac{1}{2},\hat{\pi}_{\mathbf{f}}\otimes \lambda^-)$. The Shimura variety  in question is the tower of modular curves $Sh$ classifying elliptic curves up to isogeny constructed by Shimura in \cite{Sh66} and reinterpreted by Deligne in \cite{De71}. These CM points $x =(E_x,\eta_x)\in Sh$ are associated to proper ideal classes of $R_{N_{ns}p^n}$ and carry elliptic curves $E_x$ with complex multiplication by $M$ having ordinary reduction over $\mathcal{W}$ and equipped with suitable level structures $\eta_x$ also defined over $\mathcal{W}$. 

Let $\Gamma:=1+\mathbf{p}\mathbb{Z}_p$, $\omega$ be the Teichm\"uller character, and write $\mathbf{p}$ for 4 or $p$ according to $p=2$ or not. Following recipe provided in Section \ref{sec:twovariables}, fix once and for all a finite order arithmetic Hecke character $\tilde{\psi}$ of $M$ such that $\tilde{\psi}|_{\mathbb{A}^\times}=\psi^{-1}$, as well as an arithmetic Hecke character  $\tilde{\omega}$ of $M$ of $\infty$-type $(1,0)$ such that $\tilde{\omega}|_{\mathbb{A}^\times}=\omega|\cdot|_{\mathbb{A}}$. We embedd $\mathbb{Z}_{\geq 2}$ into $\mathrm{Hom}_{cont}(\Gamma,W^\times)$ via $k\mapsto (\gamma \mapsto \gamma^k)$. Then our theorem states:

\begin{main}\label{main}
Let $\mathcal{F}=\{f_k\}_k$ be a Hida family of ordinary modular forms $f_k\in S_k^{\mathrm{ord}}(\Gamma_0(N_0p^r\mathbf{p}),\psi\omega^{-k})$, $p\nmid N$, $r\geq 0$. There exists a bounded $W$-valued $p$-adic measure $\mathrm{d}\mathcal{F}$ on $\mathrm{Cl}^-_\infty \times \Gamma$ whose Mazur--Mellin transform $\mathcal{L}(\mathcal{F};\cdot,\cdot)$ is a $p$-adic analytic Iwasawa function on the $p$-adic Lie group $\mathrm{Hom}_{cont}(\mathrm{Cl}^-_\infty\times \Gamma,W^\times)$ and such that whenever $\chi:M^{\times} \backslash M^{\times}_{\mathbb{A}}\to\mathbb{C}^\times$ is an anticyclotomic arithmetic Hecke character with $\chi(a_\infty)=a_\infty^{m(1-c)}$ for some $m\geq 0$ and of conductor $N_{ns}p^s$, where $s\geq r+ \mathrm{ord}_p(\mathbf{p})$ is an arbitrary integer, or of conductor $N_{ns}$,
we have 
\[
\left(\frac{\mathcal{L}(\mathcal{F};\widehat{\chi},k)}{\mathrm{\Omega}_p^{k+2m}}\right)^2=C(k,\tilde{\psi},\tilde{\omega},\chi,m)\frac{L(\frac{1}{2},\hat{\pi}_{\mathbf{f}_k}\otimes (\tilde{\psi}\tilde{\omega}^k\chi)^-)}{\left(\mathrm{\Omega}_\infty^{k+2m}\right)^2}
\]
for an explicit constant $C(k,\tilde{\psi},\tilde{\omega},\chi,m)$. Here $\mathrm{\Omega}_p\in W^\times$ and $\mathrm{\Omega}_\infty\in\mathbb{C}^\times$ are N\'eron periods of an elliptic curve of CM type $\Sigma$ defined over $\mathcal{W}$ (see Section \ref{sec:cmpts} for the definition), and  $\widehat{\chi}:M^{\times} \backslash M^{\times}_{\mathbb{A}}\to W^\times$ defined by $\widehat{\chi}(x)=\chi(x)x_p^{m(1-c)}$ is the $p$-adic avatar of $\chi$.
\end{main}

If we denote by $\Lambda := W[[\Gamma]]$ the Iwasawa algebra, then $\mathcal{L}(\mathcal{F};\cdot,\cdot)$ gives rise to an Iwasawa function on $\mathrm{Spec}(W[[ \mathrm{Cl}^-_\infty]]\times \Lambda)$ in the sense that it is a global section of the structure sheaf of  $\mathrm{Spec}(W[[\mathrm{Cl}^-_\infty]]\times \Lambda)$,  i.e. an element in $W[[\mathrm{Cl}^-_\infty]]\times\Lambda$. 

In Section \ref{sec:i-adic} we present  the construction in the case of a ``full'' ordinary family of modular forms living on an irreducible component $\mathrm{Spec}(\mathbb{I})\subset\mathrm{Spec}(\mathbf{h}^{\mathrm{ord}})$, for the unique ``big'' ordinary Hecke algebra $\mathbf{h}^{\mathrm{ord}}$ constructed in \cite{Hi86a} and \cite{Hi86b}, such that  $\mathrm{Spec}(\mathbb{I})$ is a finite flat irreducible covering of $\mathrm{Spec}(\Lambda)$. 
Let $\tilde{\mathbb{I}}$ be the integral closure of $\Lambda$ in the quotient field of $\mathbb{I}$. In  Section \ref{sec:i-adic} we construct an Iwasawa function on $\mathrm{Spec}(W[[ \mathrm{Cl}^-_\infty]]\times \tilde{\mathbb{I}})$ in the sense that it is a global section of the structure sheaf of  $\mathrm{Spec}(W[[\mathrm{Cl}^-_\infty]]\times \tilde{\mathbb{I}})$,  i.e. an element in $W[[\mathrm{Cl}^-_\infty]]\times \tilde{\mathbb{I}}$. 

\subsection*{Acknowledgment} This research work was started during my graduate studies under Prof. Haruzo Hida at UCLA and completed now. I would like to express my gratitude to Prof. Haruzo Hida for his generous insight, support and guidance. 

\tableofcontents

\section{Algebro-geometric and $p$-adic modular forms}

In order to set a notation, we briefly recall basics of the algebro-geometric theory of modular forms.

\subsection{Algebro-geometric modular forms}\label{sec:agmf} 
Fix a positive integer $N$ and  a base $\mathbb{Z}[\frac{1}{N}]$-algebra $B$. The modular curve $\mathfrak{M}(N)$ of level $N$ classifies pairs $(E,i_N)_{/S}$, for a $B$-scheme $S$, formed by
\begin{itemize}
\item An elliptic curve $E$ over $S$, that is, a proper smooth morphism $\pi : E\to S$ whose geometric fibers are connected curves of genus 1, together with a section $\textbf{0}:S\to E$.
\item An embedding of finite flat group schemes $i_N:\mu_N\hookrightarrow E[N]$, called level $\Gamma_1(N)$-structure,  where $E[N]$ is a scheme-theoretic kernel of multiplication by $N$ map -- it is a finite flat abelian group scheme over $S$ of rank $N^2$.
\end{itemize}
In other words, $\mathfrak{M}(N)$ is a coarse moduli scheme, and a fine moduli scheme if $N>3$, of the following functor from the category of $B$-schemes to the category \textit{SETS}
\[ \mathcal{P}_N(S)=[(E,i_N)_{/S}]_{/\cong}\,,\]
where $[\quad]_{/\cong}$ denotes the set of isomorphism classes of the objects inside the brackets. If $\omega$ is a basis of $\pi_*(\Omega_{E/S})$, that is, a nowhere vanishing section of $\Omega_{E/S}$, one can further consider a functor classifying triples $(E,i_N,\omega)_{/S}$:
\[ \mathcal{Q}_N(S)=[(E,i_N,\omega)_{/S}]_{/\cong}\; .\]
Since $a\in \mathbb{G}_m(S)$ acts on $\mathcal{Q}(S)$ via $(E,i_N,\omega)\mapsto (E,i_N,a\omega)$, $\mathcal{Q}_N$ is a $\mathbb{G}_m$-torsor over $\mathcal{P}_N$. This furnishes representability of $\mathcal{Q}$ by a $B$-scheme $\mathcal{M}(N)$ affine over $\mathfrak{M}(N)_{/B}$.

Fix a positive integer $k$ and a continuous character $\psi : (\mathbb{Z}/N\mathbb{Z})^\times\rightarrow B^\times$. Denote by $\zeta_N$ the canonical generator of $\mu_N$. A $B$-integral holomorphic modular form of weight $k$, level $\Gamma_0(N)$ and nebentypus $\psi$ is a function of isomorphism classes of $(E,i_N,\omega)_{/A}$, defined over $B$-algebra $A$, satisfying the following conditions:
\begin{itemize}
 \item[(G0)] $f((E,i_N,\omega)_{/A})\in A$ if $(E,i_N,\omega)$ is defined over $A$;
 \item[(G1)] If $\varrho:A\rightarrow A'$ is a morphism of $B$-algebras then $f((E,i_N,\omega)_{/A}\otimes_{B} A')=\varrho(f((E,i_N,\omega)_{/A}))$;
 \item[(G2)] $f((E,i_N,a\omega)_{/A})=a^{-k}f((E,i_N,\omega)_{/A})$ for $a\in A^\times=\mathbb{G}_m(A)$;
 \item[(G3)] $f((E, i_N\circ b,\omega)_{/A})=\psi(b)f((E,i_N,\omega)_{/A})$ for $b \in (\mathbb{Z}/N\mathbb{Z})^\times$, where $b$ acts on $i_N$  by the canonical action of $\mathbb{Z}/N\mathbb{Z}$ on the finite flat group scheme $\mu_N$;
 \item[(G4)] For the Tate curve $Tate(q^N)$ over $B\otimes_{\mathbb{Z}}\mathbb{Z}((q))$ viewed as algebraization of formal quotient        
   $\widehat{\mathbb{G}}_m/q^{N\mathbb{Z}}$, its canonical differential $\omega_{Tate}^{can}$ deduced from $\frac{\mathrm{d}t}{t}$  on $\widehat{\mathbb{G}}_m$,  the canonical level $\Gamma_1(N)$-structure $i_{Tate,N}^{can}$ coming from the canonical image of the point $\zeta_N$ from  $\widehat{\mathbb{G}}_m$, and all $\alpha \in \mathrm{Aut}(Tate(q^N)[N]) \cong G(\mathbb{Z}/N\mathbb{Z})$, we have
   \[ f((Tate(q^N),\alpha\circ i_{Tate,N}^{can},\omega_{Tate}^{can}))\in B\otimes_{\mathbb{Z}}\mathbb{Z}[[q]]\; .\] 
\end{itemize}
The space of $B$-integral holomorphic modular forms of weight $k$, level $\Gamma_0(N)$ and nebentypus $\psi$ is a $B$-module of finite type and we denote it by $G_k(N,\psi;B)$.

\subsection{$p$-adic modular forms}
\label{p-adicmf}
Fix a prime number $p$ that does not divide $N$. Let $B$ be an algebra that is complete and separated in its $p$-adic topology; such algebras are called $p$-adic algebras. For an elliptic curve $E_{/S}$ we consider the Barsotti--Tate group $E[p^\infty]=\underrightarrow{\lim}_nE[p^n]$, for finite flat group schemes $E[p^n]$ equipped with closed immersions $E[p^n]\hookrightarrow E[p^m]$ for $m>n$, and the multiplication $[p^{m-n}]:E[p^m]\to E[p^n]$ which is an epimorphism in the category of finite flat group schemes. Considering a morphism of ind-group schemes $i_p:\mu_{p^\infty}\hookrightarrow E[p^\infty]$ we have a functor
\[ \widehat{\mathcal{P}}_N(A)=[(E,i_N,i_p)_{/A}]_{/\cong} \]
defined over the category of $p$-adic $B$-algebras $A$. By a theorem of Deligne--Ribet and Katz, this functor is pro-represented by the formal completion $\widehat{\mathfrak{M}}(Np^\infty)$ of $\mathfrak{M}(N)$ along the ordinary locus of its modulo $p$ fiber. A holomorphic $p$-adic modular form over $B$ is a function of isomorphism classes of $(E,i_N,i_p)_{/A}$, defined over $p$-adic $B$-algebra $A$, satisfying the following conditions:
\begin{itemize}
	\item[(P0)] $f((E,i_N,i_p)_{/A})\in A$ if $(E,i_N,i_p)$ is defined over $A$;
	\item[(P1)] If $\varrho:A\rightarrow A'$ is a $p$-adically continuous morphism of $B$-algebras then $f((E,i_N,i_p)_{/A}\otimes_{B} A')=\varrho(f((E,i_N,i_p)_{/A}))$;
	\item[(P2)]  For the Tate curve $Tate(q^N)$ over $\widehat{B((q))}$, which is a $p$-adic completion of $B((q))$, the canonical $p^\infty$-structure $i_{Tate,p}^{can}$, the canonical level $\Gamma_1(N)$-structure $i_{Tate,N}^{can}$, all $p$-adic units $z\in \mathbb{Z}_p^\times$ and all $\alpha \in \mathrm{Aut}(Tate(q^N)[N]) \cong  G(\mathbb{Z}/N\mathbb{Z})$, we have 
\[ f((Tate(q^N),\alpha \circ i_{Tate,N}^{can},z\circ i_{Tate,p}^{can}))\in B[[q]] \;.\] 
\end{itemize}
We denote the space of $p$-adic holomorphic modular forms over $B$ by $V(N;B)$. 

The fundamental $q$-expansion principle holds for both algebro-geometric and $p$-adic modular forms (see Section 3.2 of \cite{Br11a} for the statement under our notation).

The $p^\infty$-level structure $i_p:\mu_{p^\infty}\hookrightarrow E[p^\infty]$ over $A$ induces an isomorphism of formal groups $\hat{i}_p:\widehat{\mathbb{G}}_m\cong\widehat{E}$ over $A$ called trivialization of $E$, where $\widehat{E}$ stands for the formal completion of $E$ along its zero-section. Using the trivialization $\hat{i}_p$ we can push forward the canonical differential $\frac{\mathrm{d}t}{t}$ on $\widehat{\mathbb{G}}_m$ to obtain an invariant differential $\omega_p:=\hat{i}_{p,*}(\frac{\mathrm{d}t}{t})$ on $\widehat{E}$ which then extends to an invariant differential on $E$. Thus for $f\in G_k(N,\psi;B)$ we can define
\[f((E,i_N,i_p)):=f((E,i_N,\omega_p))\]
and thus regard an algebro-geometric holomorphic modular form as a $p$-adic one. The $q$-expansion principle makes $G_k(N,\psi;B)\hookrightarrow V(N;B)$ an injection preserving $q$-expansions.

\section{Modular curves}
\subsection{Elliptic curves with complex multiplication}\label{sec:eccm}
Each $\mathbb{Z}$-lattice $\mathfrak{a}$ in $M$ is actually a proper ideal of the $\mathbb{Z}$-order $R(\mathfrak{a}):=\{\alpha\in R\mid \alpha\mathfrak{a}\subset\mathfrak{a}\}$  of $M$. On the other hand, every $\mathbb{Z}$-order $O$ of $M$ is of the form $O=\mathbb{Z}+cR$ for a rational integer $c$ called the conductor. The following are equivalent (see Proposition 4.11 and (5.4.2) in \cite{IAT} and Theorem 11.3 of \cite{CRT})
\begin{itemize}
\item[(1)] $\mathfrak{a}$ is $O$-projective fractional ideal
\item[(2)] $\mathfrak{a}$ is locally principal, i.e. \frenchspacing the localization at each prime is principal
\item[(3)] $\mathfrak{a}$ is a proper $O$-ideal, i.e. \frenchspacing $O=R(\mathfrak{a})$.  
\end{itemize}
This allows us to define the class group $\mathrm{Cl}_M^-(O):=\mathrm{Pic}(O)$  as the group of $O$-projective fractional ideals modulo the globally principal ideals. It is a finite group called the ring class group of conductor $c$, where $c$ stands for the conductor of $O$. 

In this paper we are concerned with orders $R_{cp^n}:=\mathbb{Z}+cp^nR$ and their ring class groups $\mathrm{Cl}^-_n:=\mathrm{Pic}(R_{cp^n})$ when $n\geq 0$, where $c$ is a fixed choice of an integer prime to $p$ that will always be clear from the context. By the class field theory, $\mathrm{Cl}^-_n$ is the Galois group $\mathrm{Gal}(H_{cp^n}/M)$ of the ring class field $H_{cp^n}$ of $M$ of conductor $cp^n$. The adelic interpretation of $\mathrm{Cl}^-_n$ is given by
\[ \mathrm{Cl}^-_n = M^\times \left\backslash (M^{(\infty)}_{\mathbb{A}})^{\times}\right/(\mathbb{A}^{(\infty)})^\times \widehat{R}_{cp^n}^{\times} \]
where $\widehat{R}_{cp^n}=R_{cp^n} \otimes_\mathbb{Z}\widehat{\mathbb{Z}}$. The anticyclotomic class group modulo $cp^\infty$ is defined as $\mathrm{Cl}^-_\infty:=\varprojlim_n\mathrm{Cl}^-_n$ for the projection $\pi_{m+n,n}:\mathrm{Cl}^-_{m+n}\to\mathrm{Cl}^-_n$ taking $\mathfrak{a}$ to $\mathfrak{a}R_{cp^n}$. It is isomorphic to the Galois group of the maximal ring class field $H_{cp^\infty}=\bigcup_n H_{cp^n}$ of conductor $cp^\infty$ of $M$.

Take a $\mathbb{Z}$-lattice $\mathfrak{a}\subset M$ having $p$-adic completion $\mathfrak{a}_p=\mathfrak{a}\otimes_\mathbb{Z}\mathbb{Z}_p$ identical to $R\otimes_\mathbb{Z}\mathbb{Z}_p$. Starting from a complex torus $\X{\mathfrak{a}}(\mathbb{C})=\mathbb{C}/\mathfrak{a}$, by the main theorem of complex multiplication (\cite{ACM} 18.6), we  algebraize it to an elliptic curve having complex multiplication by $M$ and defined over a number field. Then applying the Serre--Tate's criterion of good reduction (\cite{SeTa}) we deduce that \X{\mathfrak{a}} is actually defined over the field of fractions $\mathcal{K}$ of $\mathcal{W}$ and extends to an elliptic curve over $\mathcal{W}$ still denoted by \Xw{\mathfrak{a}}{\mathcal{W}}. All endomorphisms of \Xw{\mathfrak{a}}{\mathcal{W}} are defined over $\mathcal{W}$ and its special fiber $\Xw{\mathfrak{a}}{\bar{\mathbb{F}}_p}=\Xw{\mathfrak{a}}{\mathcal{W}}\otimes \bar{\mathbb{F}}_p$ is ordinary by our assumption that $p=\mathfrak{p}\bar{\mathfrak{p}}$ splits in $M$.

Let $T(\X{\mathfrak{a}})=\varprojlim_N\X{\mathfrak{a}}[N](\bar{\mathbb{Q}})$ be the Tate module of \X{\mathfrak{a}}. A choice of a $\widehat{\mathbb{Z}}$-basis $(w_1,w_2)$ of $\widehat{\mathfrak{a}}=\mathfrak{a}\otimes_{\mathbb{Z}}\widehat{\mathbb{Z}}$ gives rise to a level $N$-structure $\eta_N(\mathfrak{a}):(\mathbb{Z}/N\mathbb{Z})^2\cong \X{\mathfrak{a}}[N]$ given by $\eta_N(\mathfrak{a})(x,y)=\frac{xw_1+yw_2}{N}\in \X{\mathfrak{a}}[N]$. After taking their inverse limit and tensoring with $\mathbb{A}^{(\infty)}$, we get a level structure
\[ \eta(\mathfrak{a}) = \varprojlim_N\eta_N(\mathfrak{a}):(\mathbb{A}^{(\infty)})^2\cong T(\X{\mathfrak{a}})\otimes_{\widehat{\mathbb{Z}}}\mathbb{A}^{(\infty)}=:V(\X{\mathfrak{a}}) \; . \] 
We can remove the $p$-part of $\eta(\mathfrak{a})$ and define a level structure $\eta^{(p)}(\mathfrak{a})$ that conveys information about all prime-to-$p$ torsion in \X{\mathfrak{a}}:
\[ \eta^{(p)}(\mathfrak{a}):(\mathbb{A}^{(p\infty)})^2\cong
T(\X{\mathfrak{a}})\otimes_{\widehat{\mathbb{Z}}}\mathbb{A}^{(p\infty)}=:V^{(p)}(\X{\mathfrak{a}}) \; . \]
Note that prime-to-$p$ torsion in \Xw{\mathfrak{a}}{\mathcal{W}} is unramified at $p$, and $\X{\mathfrak{a}}[N]$ for $p\nmid N$ is \'etale whence constant over $\mathcal{W}$, so the level structure $\eta^{(p)}(\mathfrak{a})$ is still defined over $\mathcal{W}$ (\cite{ACM} 21.1 and \cite{SeTa}). 

Since \Xw{\mathfrak{a}}{\mathcal{W}} has ordinary reduction over $\mathcal{W}$, we can identify $\mu_{p^\infty}$ with the connected component $\X{\mathfrak{a}}[p^\infty]^\circ \cong \X{\mathfrak{a}}[\mathfrak{p}^\infty]$, obtaining the ordinary part of level structure at $p$, namely $\ord{\mathfrak{a}}:\mu_{p^\infty}\hookrightarrow\X{\mathfrak{a}}[p^\infty]$. The \'etale part of level structure at $p$, namely $\et{\mathfrak{a}}:\mathbb{Q}_p/\mathbb{Z}_p\cong\X{\mathfrak{a}}[p^\infty]^{\acute{e}t}\cong \X{\mathfrak{a}}[\bar{\mathfrak{p}}^\infty]$ over $\mathcal{W}$ is then furnished by the Cartier duality.
Thus, we constructed a triple
\[\xw{\mathfrak{a}}{\mathcal{W}}\,.\]

\subsection{Definitions and basic facts}\label{sec:dbf} 

For the affine algebraic group $G=\mathrm{GL}(2)_{/\mathbb{Q}}$ let $\mathbb{S}=\mathrm{Res}_{\mathbb{C}/\mathbb{R}}\mathbb{G}_m$ and denote by $h_{\mathbf{0}}:\mathbb{S}\to G_{/\mathbb{R}}$ the homomorphism of real algebraic groups sending $a+b\mathbf{i}$ to the matrix $\bigl(\begin{smallmatrix} a & -b \\ b & a \end{smallmatrix} \bigr)$. The symmetric domain $\mathfrak{X}$ for $G(\mathbb{R})$ can be identified with conjugacy class of $h_{\mathbf{0}}$ under $G(\mathbb{R})$ and is isomorphic to the union $\mathfrak{H}\sqcup\mathfrak{H}^{c}$ of complex upper and lower half planes via $g\circ h_{\mathbf{0}}\mapsto g\circ \mathbf{i}$. Here the left actions of $G(\mathbb{R})$ on $\mathfrak{X}$ and $\mathfrak{H}\sqcup\mathfrak{H}^{c}$ are by conjugation and $z\mapsto\frac{az+b}{cz+d}$, for $g= \bigl(\begin{smallmatrix} a & b \\ c & d \end{smallmatrix} \bigr)$, respectively. The pair $(G,\mathfrak{X})$ satisfies Deligne's axioms for having its Shimura variety $Sh$ (\cite{De71} and \cite{De79} 2.1.1). $Sh$ was first  constructed by Shimura in \cite{Sh66} but reinterpreted by Deligne in \cite{De71} 4.16-4.22 as a moduli of elliptic curves up to isogenies. 
More precisely, Deligne realized $Sh$ as a quasi-projective smooth $\mathbb{Q}$-scheme representing the moduli functor $\mathcal{F}^{\mathbb{Q}}$ from the category of abelian $\mathbb{Q}$-schemes to \textit{SETS}:
\[ \mathcal{F}^{\mathbb{Q}}(S)=\{(E,\eta)_{/S}\}_{/\approx}\,, \] 
where $\eta:(\mathbb{A}^{(\infty)})^2\cong T(E)\otimes_{\widehat{\mathbb{Z}}}\mathbb{A}^{(\infty)}=:V(E)$ is a $\mathbb{Z}$-linear isomorphism and two pairs $(E,\eta)_{/S}$ and $(E',\eta')_{/S}$ are isomorphic up to an isogeny, which we write $(E,\eta)_{/S}\approx (E',\eta')_{/S}$, if there exists an isogeny $\phi:E_{/S}\to E'_{/S}$ such that $\phi\circ\eta=\eta'$.

An important point in Deligne's treatment of $Sh_{/\mathbb{Q}}$ is that instead of the functor $\mathcal{F}^{\mathbb{Q}}$ one can consider the isomorphic functor $\widetilde{\mathcal{F}}^{\mathbb{Q}}$ from the category of abelian $\mathbb{Q}$-schemes to \textit{SETS}:
\[
\widetilde{\mathcal{F}}^{\mathbb{Q}}(S)= \{(E',\eta)_{/S}|\, \exists E\in \mathcal{F}^{\mathbb{Q}}(S)\, :\, E'_{/S}\approx E_{/S}\, ,\, \eta(\widehat{\mathbb{Z}}^2)=T(E')\}_{/\cong} \,,
\] 
and $\cong$ is not just induced from an isogeny, but rather an isomorphism of elliptic curves. Imposing the extra condition $\eta(\widehat{\mathbb{Z}}^2)=T(E')$ is compensated by tightening equivalence from ``isogenies'' to ``isomorphisms'' (see Section 4 of \cite{Br11a} for details).

The pairs $(E,\eta^{(p)})_{/S}$, for  a $\mathbb{Z}_{(p)}$-scheme $S$, consisting of an elliptic curve $E$ over $S$ and a $\mathbb{Z}$-linear isomorphism $\eta^{(p)}:(\mathbb{A}^{(p\infty)})^2\cong T(E)\otimes_{\widehat{\mathbb{Z}}}\mathbb{A}^{(p\infty)}=:V^{(p)}(E)$, are classified up to isogenies of degree prime to $p$ by a $p$-integral model $Sh_{/\mathbb{Z}_{(p)}}^{(p)}$  of  $Sh/G(\mathbb{Z}_p)$ (\cite{Ko}). By its construction, $Sh^{(p)}$ is a smooth $\mathbb{Z}_{(p)}$-scheme representing the moduli functor $\mathcal{F}^{(p)}$ from the category of  $\mathbb{Z}_{(p)}$-schemes to \textit{SETS}
\begin{equation} \label{kott1}
\mathcal{F}^{(p)}(S)=\{(X,\eta^{(p)})_{/S}\}_{/\approx}\,, 
\end{equation} 
and two pairs $(X,\eta^{(p)})_{/S}$ and $(X',\eta'^{(p)})_{/S}$ are isomorphic up to a prime-to-$p$ isogeny, which we write $(X,\eta^{(p)})_{/S}\approx (X',\eta'^{(p)})_{/S}$, if there exists an isogeny $\phi:X_{/S}\to X'_{/S}$ of degree prime to $p$ such that $\phi\circ\eta=\eta'$. By the same token, we could also tighten the equivalence from ``isogenies'' to ``isomorphisms'' and consider the isomorphic functor $\widetilde{\mathcal{F}}^{(p)}$ from the category of $\mathbb{Z}_{(p)}$-schemes to \textit{SETS}:
\begin{equation} \label{kott2}
\widetilde{\mathcal{F}}^{(p)}(S)= \{(E',\eta)_{/S}|\, \exists E\in \mathcal{F}^{(p)}(S)\, :\, E'_{/S}\approx E_{/S}\, ,\, \eta^{(p)}(\widehat{\mathbb{Z}}^2)=T^{(p)}(E')\}_{/\cong}\,.
\end{equation} 

The pair $x(\mathfrak{a})= \xunrw{\mathfrak{a}}{\mathcal{W}}$ for a $\mathbb{Z}$-lattice $\mathfrak{a}$ with $\mathfrak{a}\otimes_\mathbb{Z}\mathbb{Z}_p =R\otimes_\mathbb{Z}\mathbb{Z}_p$, constructed as above, gives rise to a $\mathcal{W}$-point on $Sh^{(p)}$ to which we refer as a CM point.

$Sh^{(p)}$ has a $G(\mathbb{A}^{(\infty)})$-action where each adele $g\in G(\mathbb{A}^{(\infty)})$ acts on a level structure $\eta^{(p)}$ by $\eta^{(p)}\mapsto\eta\circ g^{(p\infty)}$. A sheaf theoretic coset $\bar{\eta}^{(p)}=\eta^{(p)} K$, for an open compact subgroup $K\subset G(\mathbb{A}^{(\infty)})$ maximal at $p$ (i.e. $K_p=G(\mathbb{Z}_p)$), is called a level $K$-structure. The quotient $Sh^{(p)}_K=Sh^{(p)}/K$ represents the quotient functor
\[\mathcal{F}^{(p)}_K(S)=\{(E,\bar{\eta}^{(p)})_{/S}\}_{/\approx} \,, \] 
and $Sh^{(p)}=\varprojlim_K Sh^{(p)}_K$ when $K=G(\mathbb{Z}_p)\times K^{(p)}$ and $K^{(p)}$ running over open compact subgroups of $G(\mathbb{A}^{(p\infty)})$. 
When $K$ is chosen to be $\widehat{\Gamma}(N)$ for $p\nmid N$ (\cite{PAF} Section 4.2.1), $Sh_K$ is isomorphic to a, fine (when $N>3$) or coarse, moduli scheme $\mathfrak{M}(\Gamma(N))_{/\mathbb{Z}[1/N]}$ of level $\Gamma(N)$ (the principal congruence subgroup) representing the following functor
from the category of $\mathbb{Z}[1/N]$-schemes to the category \textit{SETS}
\[ \mathcal{P}_{\Gamma(N)}(S)=[(E,\phi_N:(\mathbb{Z}/N\mathbb{Z})\cong E[N])_{/S}]_{/\cong}\,,\]
and hence
\[Sh^{(p)}_{/\mathcal{W}}\cong\varprojlim_{p\nmid N} \mathfrak{M}(\Gamma(N))_{/\mathcal{W}} \,. \]
Then
\[ \mathfrak{M}(\Gamma(N))_{/\mathbb{Z}[1/N,\mu_N]}=\bigsqcup_{\zeta} \mathfrak{M}(\Gamma(N),\zeta) \,,\] 
where the union is taken over all primitive $N$-th roots of unity $\zeta$ and $\mathfrak{M}(\Gamma(N),\zeta)$ is a modular curve of level $\Gamma(N)$ representing the following functor from the category of $\mathbb{Z}[1/N,\mu_N]$-schemes to the category \textit{SETS}
\[ \mathcal{P}_{\Gamma(N), \zeta}(S)=[(E,\phi_N:(\mathbb{Z}/N\mathbb{Z})\cong E[N])_{/S}\mid \langle \phi_N(1,0),\phi_N(0,1)\rangle=\zeta]_{/\cong}\,,\]
with $\langle \cdot,\cdot\rangle$ denoting the Weil pairing, and with $\mathfrak{M}(\Gamma(N),\zeta)(\mathbb{C})=\Gamma(N)\backslash \mathfrak{H}$. 

The complex points of  $Sh$ have the following expression
\[ Sh(\mathbb{C})= G(\mathbb{Q})\left\backslash \left( \mathfrak{X}\times G(\mathbb{A}^{(\infty)})\right) \right/ Z(\mathbb{Q})\,,\]
where $Z$ denotes the center of $G$ and the action is given by $\gamma(z,g)u=(\gamma(z),\gamma gu)$ for $\gamma \in G(\mathbb{Q})$ and $u\in Z(\mathbb{Q})$ (\cite{De79} Proposition 2.1.10 and \cite{Mi} page 324 and Lemma 10.1). We write $[z,g]\in Sh(\mathbb{C})$ for the image of $(z,g)\in \mathfrak{X}\times G(\mathbb{A}^{(\infty)})$. 

\subsection{CM points} \label{sec:cmpts}
A point $x=[z,g]\in Sh(\mathbb{C})$ where $z$ generates an imaginary quadratic field $M_x=\mathbb{Q}(z)$ is called a CM point. In this paper we devote our attention to points $x=[z,g]$ for which $M_x=\mathbb{Q}(z)$ is actually our fixed imaginary quadratic field $M$. The corresponding elliptic curve $E_x$ from $x=(E_x,\eta_x)$ has a CM type $(M,\Sigma)$.

We define a regular representation $\rho=\rho_{z}:M^\times \hookrightarrow G(\mathbb{Q})$ by $\begin{pmatrix} \alpha z \\ \alpha \end{pmatrix}=\rho_{z}(\alpha)\begin{pmatrix} z \\ 1 \end{pmatrix}$.
Tensoring with $\mathbb{A}^{(\infty)}$, we may regard $\rho_{z}$ as a representation $\hat{\rho}_{z}:(M^{(\infty)}_{\mathbb{A}})^{\times}\hookrightarrow G(\mathbb{A}^{(\infty)})$, and further conjugating by $g$ we get $\hat{\rho}_{x}:(M^{(\infty)}_{\mathbb{A}})^{\times}\hookrightarrow G(\mathbb{A}^{(\infty)})$ given by $\hat{\rho}_{x}=g^{-1}\hat{\rho}_{z}(\alpha)g$. To each point $(E,\eta)\in Sh$ we can associate a lattice $\widehat{L}=\eta^{-1}(T(E))\subset (\mathbb{A}^{(\infty)})^2$. In the view of a basis $w$ of $\widehat{L}$, the $G(\mathbb{A}^{(\infty)})$-action on $Sh$ given by $(E,\eta)\mapsto(E,\eta\circ g)$ is the matrix multiplication $w^\intercal\mapsto g^{-1}w^\intercal$ because $(\eta\circ g)^{-1}(T(E))=g^{-1}\eta^{-1}(T(E))=g^{-1}\widehat{L}$, where $\intercal$ stands for the transpose. The fiber $E_x$ at $x\in Sh(\mathbb{C})$ of the universal elliptic curve over $Sh_{/\mathbb{Q}}$ has complex multiplication by an order of $M$, that is, under the action of $\widehat{R}$ via $\hat{\rho}_{x}$, $g^{-1}\widehat{L}\cap\mathbb{Q}^2$ is identified with a fractional ideal of $M$ prime to $p$.

The CM points associated to proper $R_{cp^n}$-ideals $\mathfrak{a}$ prime to $p$ are built in three stages. The full account of this construction is given in Section 5.1 of \cite{Br11a}, we refer the reader there for details beyond the following brief exposition. 

First, we equip \Xw{R}{\mathcal{W}} with appropriate level structures and a fixed choice of an invariant differential $\omega(R)$ on \Xw{R}{\mathcal{W}} so that $H^0(\X{R},\Omega_{\Xw{R}{\mathcal{W}}})=\mathcal{W}\omega(R)$.  Any choice of a $\widehat{\mathbb{Z}}$-basis $(w_1,w_2)$ of $\widehat{R}=R\otimes_{\mathbb{Z}}\widehat{\mathbb{Z}}$ gives rise to a level structure $\eta^{(p)}(R):(\mathbb{A}^{(p\infty)})^2\cong V^{(p)}(\X{R})$ defined over $\mathcal{W}$, as explained in Section \ref{sec:eccm}. In particular, fixing a choice of $z_1\in R$ such that $R=\mathbb{Z}+\mathbb{Z}z_1$ we get a $\widehat{\mathbb{Z}}$-basis $(w_1,w_2)$ of $\widehat{R}$. Since \Xw{R}{\mathcal{W}} has ordinary reduction over $\mathcal{W}$, we can identify $\mu_{p^\infty}$ with the connected component $\X{R}[p^\infty]^\circ \cong \X{R}[\mathfrak{p}^\infty]$, obtaining the ordinary part of a level structure at $p$, namely $\ord{R}:\mu_{p^\infty}\hookrightarrow\X{R}[p^\infty]$. The \'etale part of a level structure at $p$, namely $\et{R}:\mathbb{Q}_p/\mathbb{Z}_p\cong\X{R}[p^\infty]^{\acute{e}t}\cong \X{R}[\bar{\mathfrak{p}}^\infty]$ over $\mathcal{W}$ is then furnished by the Cartier duality. 

Second, for every proper $R_c$-ideal $\mathfrak{A}$ such that $\mathfrak{A}_p=R\otimes_\mathbb{Z}\mathbb{Z}_p$, we then induce from \Xw{R}{\mathcal{W}}  corresponding level structures and an invariant differential on \Xw{\mathfrak{A}}{\mathcal{W}}.  Regarding $c$ as an element of $\mathbb{A}^\times$, $(cw_1,w_2)$ is a basis of $\widehat{R}_c$ over $\widehat{\mathbb{Z}}$ giving rise to a level structure $\eta^{(p)}(R_c):(\mathbb{A}^{(p\infty)})^2\cong V^{(p)}(\X{R_c})$. Choosing a complete set of representatives $\{a_1,\ldots,a_{H^-}\}\subset M^{\times}_{\mathbb{A}}$ so that $M^{\times}_{\mathbb{A}}=\bigsqcup_{j=1}^{H^-}M^\times a_j\widehat{R}_c^\times M_\infty^\times$, we have $\widehat{\mathfrak{A}}=\alpha a_j\widehat{R}_c$ for some $\alpha\in M^\times$ and $1\leq j\leq H^-$, and we can define $\eta^{(p)}(\mathfrak{A})=\alpha^{-1} a_j^{-1} \eta^{(p)}(R_c)$.
Since $\mathfrak{A}_p=R\otimes_\mathbb{Z}\mathbb{Z}_p$, \X{R\cap\mathfrak{A}} is an \'etale covering of both \X{\mathfrak{A}} and \X{R}, and we get $\ord{\mathfrak{A}}:\mu_{p^\infty}\cong \X{\mathfrak{A}}[p^\infty]^\circ$ and $\et{\mathfrak{A}}:\mathbb{Q}_p/\mathbb{Z}_p\cong\X{\mathfrak{A}}[p^\infty]^{\acute{e}t}$ first by pulling back \ord{R} and \et{R} from  \X{R} to \X{R\cap\mathfrak{A}} and then by push-forward from \X{R\cap\mathfrak{A}} to \X{\mathfrak{A}}. Thus we get a CM point \[x(\mathfrak{A})=\xw{\mathfrak{A}}{\mathcal{W}}\] 
on $Sh$ associated to a proper $R_c$-ideal $\mathfrak{A}$. Note that $\omega(R)$ induces a differential $\omega(\mathfrak{A})$ on \X{\mathfrak{A}} first by pulling back $\omega(R)$ from  \X{R} to \X{R\cap\mathfrak{A}} and then by pull-back inverse from \X{R\cap\mathfrak{A}} to \X{\mathfrak{A}}. The projection $\pi_1:\X{R\cap\mathfrak{A}}\twoheadrightarrow\X{\mathfrak{A}}$ is \'etale so the pull-back inverse $(\pi_1^*)^{-1}:\Omega_{\X{R\cap\mathfrak{A}}/\mathcal{W}}\to\Omega_{\X{\mathfrak{A}}/\mathcal{W}}$ is an isomorphism, whence $H^0(\X{\mathfrak{A}},\Omega_{\X{\mathfrak{A}}/\mathcal{W}})=\mathcal{W}\omega(\mathfrak{A})$.

Third, if $C\subset \X{\mathfrak{A}}[p^n]$ is a suitable rank $p^n$ subgroup scheme of a finite flat group scheme $\X{\mathfrak{A}}[p^n]$ that is \'etale locally isomorphic to $\mathbb{Z}/p^n\mathbb{Z}$, we study geometric quotient of \X{\mathfrak{A}} by $C$ giving rise to a desired CM point $x(\mathfrak{a})$ associated to a proper $R_{cp^n}$-ideal $\mathfrak{a}$ prime to $p$. To be more precise, note that $\X{\mathfrak{A}}[p^n]=\X{\mathfrak{A}}[\bar{\mathfrak{p}}^n]\oplus\X{\mathfrak{A}}[\mathfrak{p}^n]=\mathbb{Z}/p^n\mathbb{Z}\oplus \mu_{p^n}$ over $\mathcal{W}$. If $\zeta_{p^n}$ and $\gamma_{p^n}$ are the canonical generators of $\mu_{p^n}$ and $\mathbb{Z}/p^n\mathbb{Z}$, respectively, then we actually consider $C$ to be one of the $p^{n-1}(p-1)$ rank $p^n$ finite flat subgroup schemes $C_u=\langle\zeta_{p^n}^{-u}\gamma_{p^n}\rangle$ of $\X{\mathfrak{A}}[p^n]$, for $1\leq u\leq p^n$ and $\mathrm{gcd}(u,p)=1$. Since the extension $W/\mathbb{Z}_p$ is unramified, $C_u$ as a finite flat subgruop scheme is well defined over $\mathcal{W}[\mu_{p^n}]$. Thus, the geometric quotient $\X{\mathfrak{A}}/C_u$ is defined over $\mathcal{W}[\mu_{p^n}]$. It can be shown that $\X{\mathfrak{A}}/C_u\cong\X{\mathfrak{a}_u}$ where $\mathfrak{a}_u$ is a representative of one of $p^{n-1}(p-1)$ proper $R_{cp^n}$-ideal classes in $\mathrm{Cl}_n^-$ that project to the proper ideal class of $\bar{\mathfrak{p}}^{-n}\mathfrak{A}$ in $\mathrm{Cl}_0^-$. 

The quotient map $\pi:\X{R_c}\twoheadrightarrow \X{\mathfrak{a}_u}$ is \'etale over $\mathcal{W}[\mu_{p^n}]$ so we obtain level structures $\eta^{(p)}(\mathfrak{a}_u)=\pi_{*}\eta^{(p)}(\mathfrak{A})=\pi\circ\eta^{(p)}(\mathfrak{A})$, $\ord{\mathfrak{a}}=\pi_{*}\ord{\mathfrak{A}}=\pi\circ\ord{\mathfrak{A}}$ and $\et{\mathfrak{a}}=\pi_{*}\et{\mathfrak{A}}=\pi\circ\et{\mathfrak{A}}$, as well as an invariant differential $\omega(\mathfrak{a})=(\pi^*)^{-1}\omega(\mathfrak{A})$ on $\X{\mathfrak{a}_u}$. In this way we created $p^{n-1}(p-1)$ CM points
\[x(\mathfrak{a}_u)=\xw{\mathfrak{a}_u}{\mathcal{W}[\mu_{p^n}]}\]
on $Sh$, equipped with an invariant differential $\omega(\mathfrak{a}_u)$ on \X{\mathfrak{a}_u}, where these $\mathfrak{a}_u$'s are representatives of exactly $p^{n-1}(p-1)$ proper $R_{cp^n}$-ideal classes in $\mathrm{Cl}_n^-$ that project to the proper ideal class of $\bar{\mathfrak{p}}^{-n}\mathfrak{A}$ in $\mathrm{Cl}_0^-$.

Note that the complex uniformization $\X{\mathfrak{a}_u}(\mathbb{C})=\mathbb{C}/\mathfrak{a}_u$ induces a canonical invariant differential $\omega_\infty(\mathfrak{a}_u)$ in $\Omega_{\X{\mathfrak{a}_u}/\mathbb{C}}$ by pulling back $\mathrm{d}u$, where $u$ is the standard variable on $\mathbb{C}$. Then one can define a period $\Omega_\infty\in\mathbb{C}^\times$ by $\omega(\mathfrak{a}_u)=\Omega_\infty\omega_\infty(\mathfrak{a}_u)$ (\cite{Ka} Lemma 5.1.45). Note that $\Omega_\infty$ does not depend on $\mathfrak{a}_u$ since $\omega(\mathfrak{a}_u)$ is induced by $\omega(R)$ on \X{R} by construction. 

As explained at the end of Section \ref{p-adicmf}, the ordinary part of the level structure at $p$, $\ord{\mathfrak{a}_u}:\mu_{p^\infty}\cong \X{\mathfrak{a}_u}[\mathfrak{p}^\infty]$ induces a trivialization $\widehat{\mathbb{G}}_m\cong \widehat{\X{\mathfrak{a}_u}}$ for the $p$-adic formal completion $\widehat{\X{\mathfrak{a}_u}}_{/W[\mu_{p^n}]}$ of \X{\mathfrak{a}_u} along its zero-section. We obtain an invariant differential $\omega_p(\mathfrak{a}_u)$ on $\widehat{\X{\mathfrak{a}_u}}_{/W[\mu_{p^n}]}$ by pushing forward $\frac{\mathrm{d}t}{t}$ on $\widehat{\mathbb{G}}_m$, which then extends to an invariant differential on \Xw{\mathfrak{a}_u}{W[\mu_{p^n}]} also denoted by $\omega_p(\mathfrak{a}_u)$. Then one can define a period $\Omega_p\in W^\times$, independent of $\mathfrak{a}_u$, by $\omega(\mathfrak{a}_u)=\Omega_p \omega_p(\mathfrak{a}_u)$ (\cite{Ka} Lemma 5.1.47).

\section{Serre--Tate deformation space}
In this section we restrict our attention to CM points  $x=[z,g]\in Sh(\mathbb{C})$ for which $z$ generates our fixed imaginary quadratic field $M$, i.e. $M=\mathbb{Q}(z)$, and so the fiber $E_x$ at $x$ of the universal elliptic curve over $Sh_{/\mathbb{Q}}$ has a CM type $(M,\Sigma)$.

Following Katz's exposition \cite{Ka78}, we recall some basics of the deformation theory of elliptic curves over complete local $W$-algebras whose residue field is $\bar{\mathbb{F}}_p$. Denote by $CL_{/W}$ the category of such algebras. 
Fix a CM point $x=(E_x,\eta^{(p)}_x, \eta^{\mathrm{ord}}_x\times \eta^{\mathrm{\acute{e}t}}_x)\in Sh$. The Serre--Tate deformation space $\widehat{S}$ represents the functor $\widehat{\mathcal{P}}:CL_{/W}\to SETS$ given by
\begin{equation}\label{stdef}
\widehat{\mathcal{P}}(A)=\{E_{/A}|E\otimes_A\bar{\mathbb{F}}_p=E_{x/\bar{\mathbb{F}}_p}\}_{/\cong}\;.
\end{equation}
First, if $\widehat{Sh}^{(p)}_{x/W}$ is a formal completion of $Sh^{(p)}_{/\mathcal{W}}$ along $x=(E_x,\eta^{(p)}_x)\in Sh^{(p)}(\bar{\mathbb{F}}_p)$, then the universality of  $Sh^{(p)}$ furnishes $\widehat{Sh}^{(p)}_{x/W}\cong\widehat{S}_{/W}$. Indeed, $\widehat{Sh}^{(p)}_{x}$ classifies $(E,\eta^{(p)}_E)_{/A}$ with $(E,\eta^{(p)}_E)\otimes_{A}\bar{\mathbb{F}}_p=(E_x,\eta^{(p)}_x)$, and the level structure $\eta^{(p)}_x$ at the special fiber extends  uniquely to $\eta^{(p)}_E$ on $E_{/A}$ because $E[N]$ (for $N$ prime to $p$) is \'etale over $\mathrm{Spec}(A)$. Second, the Serre--Tate deformation theory furnishes a canonical isomorphism $\widehat{S}_{/W}\cong\widehat{\mathbb{G}}_{m/W}$. Indeed, $E_{/A}\in\widehat{\mathcal{P}}(A)$ is determined by the extension class of connected component--\'etale quotient exact sequence of Barsotti--Tate groups
\begin{equation}\label{cceq}
0\longrightarrow E[p^\infty]^\circ \stackrel{i_E}{\longrightarrow} E[p^\infty] \stackrel{i^*_E}{\longrightarrow} E[p^\infty]^{\acute{e}t} \longrightarrow 0\,,
\end{equation}
and by a theorem of Serre and Tate (Theorem 2.3 in \cite{minv}) such an extension class over $A$ is classified by
\[\mathrm{Hom}( E[p^\infty]^{\acute{e}t}_{/A},E[p^\infty]^\circ_{/A})\cong\mathrm{Hom}({\mathbb{Q}_p/\mathbb{Z}_p}_{/A},{\mu_{p^\infty}}_{/A})= \varprojlim_n\mu_{p^n}(A)=\widehat{\mathbb{G}}_m(A)\,.\]
Note that to make this identification we used the fixed $\eta^{\mathrm{ord}}_x:\mu_{p^\infty}\cong E_x[p^\infty]^\circ$ and its Cartier dual inverse $\eta^{\mathrm{\acute{e}t}}_x:\mathbb{Q}_p/\mathbb{Z}_p\cong E_x[p^\infty]^{\acute{e}t}$.

Let $t$ be the canonical coordinate of the Serre--Tate deformation space $\widehat{S}$ so that \[\widehat{S}\cong\widehat{\mathbb{G}}_m=\mathrm{Spf}(\varprojlim_n W[t,t^{-1}]/(t-1)^n)=\mathrm{Spf}(W[[T]])\quad (T=t-1)\; .\]
Strictly speaking, the Serre--Tate coordinate $t=t_x$ depends on the point $x\in Sh$  since we used the fixed $\eta^{\mathrm{ord}}_x$ and $\eta^{\mathrm{\acute{e}t}}_x$ to identify $\widehat{S}\cong\widehat{\mathbb{G}}_m$. 

For any deformation $E_{/A}$ of  $E_x$, where $A\in CL_{/W}$, we can compute the assigned value of the Serre--Tate coordinate $t(E_{/A})\in\widehat{\mathbb{G}}_m(A)$ in the following way (see Sections 2.1-2.3 of \cite{minv} for details). Noting that any object in $CL_{/W}$ is a projective limit of artinian objects, we may assume for the moment that $A$ is an artinian object in  $CL_{/W}$. Using $\ord{E}$ and $\et{E}$, the connected component--\'etale quotient exact sequence of Barsotti--Tate groups (\ref{cceq}) becomes
\[ 0\longrightarrow \mu_{p^\infty} \stackrel{i_E}{\longrightarrow} E[p^\infty] \stackrel{i^*_E}{\longrightarrow} \mathbb{Q}_p/\mathbb{Z}_p\longrightarrow 0 \;. \]
The Drinfeld's theorem in deformation theory (Theorem 2.1 in \cite{minv}) assures that $E^\circ(A)$ is killed by $p^{n_0}$ for sufficiently large $n_0$. Starting from $y\in E(\bar{\mathbb{F}}_p)$, we can always lift it to $\tilde{y}\in E(A)$ because of smoothness of $E_{/A}$, and the lift $\tilde{y}$ is determined modulo $\mathrm{Ker}(E(A)\rightarrow E(\bar{\mathbb{F}}_p))= E^\circ$ which is a subgroup of $E[p^n]$ if $n\geq n_0$. Thus $p^n\tilde{y}\in E^\circ(A)$ is uniquely determined by $y\in E(\bar{\mathbb{F}}_p)$, and if $y\in E[p^n]$ then $p^n\tilde{y}\in E^\circ(A)$, giving rise to a homomorphism $``{p^n}":E[p^n](\bar{\mathbb{F}}_p)\rightarrow E^\circ(A)$, via $y\mapsto p^n\tilde{y}$, called the Drinfeld lift of multiplication by $p^n$ (\cite{Ka78} Lemma 1.1.2). Taking $1\in\mathbb{Z}_p\cong T(\mathbb{Q}_p/\mathbb{Z}_p)$ and viewing it as $1=\varprojlim_n\frac{1}{p^n}$ for $\frac{1}{p^n}\in\mathbb{Q}_p/\mathbb{Z}_p[p^n]$, the value $``{p^n}" i_E^{*-1}(\frac{1}{p^n})\in \mu_{p^n}(A)$ becomes stationary if $n\geq n_0$ and we have
\[t(E_{/A})=\varprojlim_n ``{p^n}"i_E^{*-1}\left(\frac{1}{p^n}\right)\in \varprojlim_n \mu_{p^n}(A)=\widehat{\mathbb{G}}_m(A)\;.\]
When $E$ is a deformation over $A \in CL_{/W}$ for $A$ not being artinian, writing $A=\varprojlim_B B$ for artinian $B$, we define $t(E_{/A})= \varprojlim_B t(E\times_A B_{/B})$.
Note that $t(E_{/A})=q_{E/A}(1,1)$ where $q_{E/A}(\cdot,\cdot):TE[p^\infty]^{\acute{e}t}\times TE[p^\infty]^{\acute{e}t}\rightarrow \widehat{\mathbb{G}}_m(A)$ is a bilinear form in Section 2 of \cite{Ka78} that corresponds to deformation $E_{/A}$ in establishing representability of the deformation functor $\widehat{\mathcal{P}}$ by the formal torus $\mathrm{Hom}_{\mathbb{Z}_p}(TE[p^\infty]^{\acute{e}t}\times TE[p^\infty]^{\acute{e}t},\widehat{\mathbb{G}}_m)$.

By definition, $t(E_x)=1$ because the connected component--\'etale quotient exact sequence of $E_x[p^\infty]$ splits over $\mathcal{W}$ by complex multiplication and the existence of a section $s$
\[\xymatrix@1{ 0 \ar[r] & \mu_{p^\infty} \ar[r]^-{i_x} & E_x[p^\infty] \ar[r]^-{i^*_x} & \mathbb{Q}_p/\mathbb{Z}_p \ar@/^1pc/[l]^-{s} \ar[r] & 0
}\]
allows us to verify
\[t(E_x)=\varprojlim_n ``{p^n}"i_x^{*-1}\left(\frac{1}{p^n}\right)=\varprojlim_n p^n s \left(\frac{1}{p^n}\right)= \varprojlim_n 1= 1 \,. \]

Let $(\boldsymbol{\mathcal{E}},\boldsymbol{\eta})$ be the universal deformation of $E_x$ over $\widehat{S}\cong\widehat{\mathbb{G}}_m$. For each $p$-adic modular form $f\in V(N;W)$ we call the expansion
\[f(t):=f(\boldsymbol{\mathcal{E}},\boldsymbol{\eta})\in W[[T]]\quad (T=t-1)\]
a $t$-expansion of $f$ with respect to the Serre--Tate coordinate around a CM point $x\in Sh$. We have the following $t$-expansion principle:
\[\tag{t-exp} \text{The }t\text{-expansion: }f\mapsto f(t)\in W[[T]] \text{ determines }f\text{ uniquely.}\]
Let $d: V(N;W)\to V(N;W)$ be the Katz $p$-adic differential operator acting on $p$-adic modular forms whose effect  on the $q$-expansion of a modular form is given by
\begin{equation}\label{diffqexp}
d \sum_{n\geq 0}a(n,f)q^n = \sum_{n\geq0}na(n,f)q^n
\end{equation}
(\cite{Ka} (2.6.27)).
A key property of $d$ is that, by its construction (\cite{Ka78} Section 4.3.1), it is $\widehat{S}$-invariant and the canonical Serre--Tate coordinate $t$ is normalized so that 
\begin{equation}\label{katzdiff}
d=t\frac{\mathrm{d}}{\mathrm{d}t} \; .
\end{equation}
Note that the $t$-expansion of $f$ with respect to the Serre--Tate coordinate $t$ around a CM point $x$ can be computed as the Taylor expansion of $f$ with respect to the variable $T$ by first applying $\frac{\mathrm{d}}{\mathrm{d}T}$ and then evaluating the result at $x$ (see (4.6) of \cite{minv}).

\section{Anticyclotomic $p$-adic L-function}\label{sec:acycpadicL}

To any given $W$-valued $p$-adic measure $\mu$ on $\mathbb{Z}_p$ we can associate the corresponding formal power series $\Phi_\mu(t)$ by
\[\Phi_\mu(t)=\sum_{n=0}^\infty \left(\int_{\mathbb{Z}_p}\binom{x}{n}\mathrm{d}\mu(x)\right)T^n \quad (T=t-1)\,.\]
It is well known that the measure $\mu$ is determined by the formal power series $\Phi_\mu(t)$ and $\mu\mapsto\Phi_\mu$ induces an isomorphism $\mathscr{M}(\mathbb{Z}_p;W)\cong W[[T]]$ between the space $\mathscr{M}(\mathbb{Z}_p;W)$ of all $W$-valued measures on $\mathbb{Z}_p$ and the ring of formal power series $W[[T]]$. Moreover
\[
\int_{\mathbb{Z}_p}x^m\mathrm{d}\mu=\left(t\frac{\mathrm{d}}{\mathrm{d}t}\right)^m\Phi_\mu|_{t=1}\quad \text{for all}\, m\geq0\,.
\]
The space of measures $\mathscr{M}(\mathbb{Z}_p;W)$ is naturally a module over the ring $\mathrm{Cont}(\mathbb{Z}_p;W)$ of continuous $W$-valued functions on $\mathbb{Z}_p$ in the following way: for $\phi\in\mathrm{Cont}(\mathbb{Z}_p;W)$ and $\mu \in \mathscr{M}(\mathbb{Z}_p;W)$
\[ \int_{\mathbb{Z}_p}\xi\mathrm{d}(\phi\mu)=\int_{\mathbb{Z}_p}\xi(x)\phi(x)\mathrm{d}\mu(x)\quad \text{for all}\, \xi\in\mathrm{Cont}(\mathbb{Z}_p;W)\,. \]
When $\phi\in\mathrm{Cont}(\mathbb{Z}_p;W)$ is a locally constant function factoring through $\mathbb{Z}_p/p^n\mathbb{Z}_p$, we have 
\begin{equation}\label{fourier}
\Phi_{\phi\mu}=[\phi]\Phi_\mu \text{ where }[\phi]\Phi_\mu(t)=p^{-n}\sum_{b\in\mathbb{Z}/p^n\mathbb{Z}}\phi(b)\sum_{\zeta\in\mu_{p^n}}\zeta^{-b}\Phi_\mu(\zeta t)
\end{equation}
and
\begin{equation}\label{twmoment}
\int_{\mathbb{Z}_p}\phi(x)x^m\mathrm{d}\mu(x)=\left(t\frac{\mathrm{d}}{\mathrm{d}t}\right)^m([\phi]\Phi_\mu)|_{t=1}\;.
\end{equation}

Let $\mathfrak{A}$ be a proper $R_{N_{ns}}$-ideal prime to $p$ representing a class in $\mathrm{Cl}^-_0$. Let $d: V(N;W)\to V(N;W)$ be the Katz $p$-adic differential operator acting on $p$-adic modular forms. Classical modular forms are defined over a number field, so we may assume that $f$ is defined over a localization $\mathcal{V}$ of the ring of integers of a certain number field. We take a finite extension of $W$ generated by $\mathcal{V}$ and, abusing the symbol, we keep denoting it $W$. Then $(d^mf)(x(\mathfrak{A}),\omega_p(\mathfrak{A}))\in W$ by the rationality result of Katz (\cite{Ka} Theorem 2.6.7). We define a $W$-valued measure $\mathrm{d}\mu_{f,\mathfrak{A}}$ on $\mathbb{Z}_p$ by
\begin{equation} \label{katzmeasure}
\int_{\mathbb{Z}_p}\binom{x}{n}\mathrm{d}\mu_{f,\mathfrak{A}}(x)=\binom{d}{n}f(x(\mathfrak{A}),\omega_p(\mathfrak{A})) \quad \text{for all}\, n\geq0\,. 
\end{equation}
Then clearly
\[ \int_{\mathbb{Z}_p}x^m\mathrm{d}\mu_{f,\mathfrak{A}} = (d^mf)(x(\mathfrak{A}),\omega_p(\mathfrak{A})) \quad \text{for all}\, m\geq0\,. \]

From the point of view of the Serre--Tate deformation theory, the power series $\Phi_{\mu_{f,\mathfrak{A}}}(t)$ has a neat description given by the following
\begin{taylor}\label{taylor}
$\Phi_{\mu_{f,\mathfrak{A}}}(t)$ is the $t$-expansion of $f$ with respect to the Serre--Tate coordinate $t$ around the point $x(\mathfrak{A})$.
\end{taylor}
\begin{proof}
The proposition follows from the elementary identity $\binom{t\frac{\mathrm{d}}{\mathrm{d}t}}{n}=\frac{t^n}{n!}\frac{\mathrm{d}^n}{\mathrm{d}t^n}$ (\cite{LFE} Lemma 3.4.1 on page 80), (\ref{katzdiff}) and the fact that $t(\X{\mathfrak{A}})=1$. Indeed, if $T=t-1$, we have
\[f(t)=\sum_{n=0}^\infty \frac{1}{n!}\frac{\mathrm{d}^nf}{\mathrm{d}t^n}(1)T^n = \sum_{n=0}^\infty \binom{t\frac{\mathrm{d}}{\mathrm{d}t}}{n}f(1)T^n= \sum_{n=0}^\infty \binom{d}{n}f(x(\mathfrak{A}),\omega_p(\mathfrak{A}))T^n= \sum_{n=0}^\infty \left( \int_{\mathbb{Z}_p}\binom{x}{n}\mathrm{d}\mu_{f,\mathfrak{A}}\right)T^n = \Phi_{\mu_{f,\mathfrak{A}}}(t).\]
\end{proof}

One of the key facts regarding the measure $\mathrm{d}\mu_{f,\mathfrak{A}}$ is given in Proposition 8.3 of \cite{Br11a}. This proposition states that for any function $\phi:\mathbb{Z}/p^n\mathbb{Z}\to\mathbb{C}$, $n\geq 1$,  we have 
\begin{equation} \label{key}
[\phi]\Phi_{\mu_{f,\mathfrak{A}}}=\Phi_{\mu_{f|\phi^-,\mathfrak{A}}}
\end{equation}
where $\phi^-:\mathbb{Z}/p^n\mathbb{Z}\to\mathbb{C}$ is defined by $\phi^-(x):=\phi(-x)$. Moreover, if $\{\mathfrak{a}_u\}_{u\in (\mathbb{Z}/p^n\mathbb{Z})^\times}$ is the set of proper $R_{N_{ns}p^n}$-ideals whose classes in $\mathrm{Cl}_n^-$ project down to the class of $\mathfrak{A}$ in $\mathrm{Cl}^-_{0}$ via the canonical $\pi:\mathrm{Cl}^-_{n}\rightarrow \mathrm{Cl}^-_{0}$, and if $\phi : (\mathbb{Z}/p^n\mathbb{Z})^\times\to\mathbb{C}^\times$ is a primitive Dirichlet character of conductor $p^n$, 
we have
\begin{equation} \label{integration}
\int_{\mathbb{Z}_p}\phi(x)x^m\mathrm{d}\mu_{f,\mathfrak{A}}(x) = G(\phi)\sum_{u\in (\mathbb{Z}/p^n\mathbb{Z})^\times}\phi^{-1}(-u)d^mf(x(\mathfrak{a}_u),\omega_p(\mathfrak{a}_u))
\end{equation}
where $G(\phi)$ is the Gauss sum.

The Proposition 8.3 of \cite{Br11a} is proved by studying the effect of the Hecke operators on the Serre--Tate coordinate $t$ in the infinitesimal neighborhoods of the CM points on $Sh$. The Hecke action in the infinitesimal neighborhood of a CM point $x$ is two-fold: it affects the Serre--Tate coordinate $t$ but also the global level structure $\eta_x^{(p)}$ of a point on $Sh$ (see \cite{Br11a} for details, and Proposition 6.4 and Proposition 7.3 in particular).
Note that for our purpose it suffices to work over Igusa tower $Ig_{N^{(p)}}$ over $\mathfrak{M}(N^{(p)})$, where $N^{(p)}$ is prime-to-$p$ level. Thus, in characteristic 0, the modular form is on $\mathfrak{M}(N^{(p)}p^r)$ for some $r\geq 0$.
In a nutshell, we first verified there that $x(\mathfrak{a}_u)= x(\mathfrak{A})\circ \bigl(\begin{smallmatrix} 1 & up^{-n} \\ 0 & 1 \end{smallmatrix} \bigr)$ on $Sh/\widehat{\Gamma}_1(N^{(p)}p^r)$ for a proper $R_{N_{ns}p^n}$-ideal  $\mathfrak{a}_u$ such that $\mathfrak{a}_uR_{N_{ns}}=\mathfrak{A}R_{N_{ns}}$. Then we computed  the Serre--Tate coordinate around $x(\mathfrak{A})$ of the point $x(\mathfrak{a}_u)$ on $Sh$, namely $t(x(\mathfrak{a}_u))=\zeta^u_{p^n}$, for a fixed choice of a primitive $p^n$-th root of unity $\zeta_{p^n}$. That being said, (\ref{integration}) follows from Proposition \ref{taylor} combined with (\ref{fourier}) and (\ref{twmoment}).

Let $\{\mathfrak{A}_1,\ldots,\mathfrak{A}_{H^-}\}$ be a complete set of representatives for $\mathrm{Cl}^-_0$ and let $\widehat{\mathfrak{A}}_j=A_j\widehat{R}_{N_{ns}}$ for $A_j\in M^{\times}_{\mathbb{A}}$, $j=1,\ldots,H^-$. We may assume that $A_{j,p}=1$ for all $j=1,\ldots,H^-$. An explicit coset decomposition 
\begin{equation}\label{coset}
\mathrm{Cl}^-_\infty=\bigsqcup_{j=1}^{H^-}\mathfrak{A}_j^{-1}\mathbb{Z}_p^\times
\end{equation}
enables us to construct a $W$-valued measure $\mathrm{d}\mu_f$ on $\mathrm{Cl}^-_\infty$ by piecing together $H^-$ distinct $W$-valued measures $\mathrm{d}\mu_j$ on $\mathbb{Z}_p^\times$, so that for every continuous function $\xi:\mathrm{Cl}^-_\infty\to W$ we have 
\[\int_{\mathrm{Cl}^-_\infty} \xi\mathrm{d}\mu_f=\sum_{j=1}^{H^-}\int_{\mathbb{Z}_p^\times}\xi(\mathfrak{A}_j^{-1}x)\mathrm{d}\mu_j(x)\,.\]

To obey technical assumptions of Hida's Waldspurger-type of formula, namely the Theorem 4.1 from \cite{HidaNV}, we provided a recipe in Section 8.2 of \cite{Br11a} on how to fix once and for all an arithmetic Hecke character $\lambda:M^{\times} \backslash M^{\times}_{\mathbb{A}}\to\mathbb{C}^\times$ so that $\lambda(a_\infty)=a_\infty^k$ and $\lambda|_{\mathbb{A}^\times}=\boldsymbol{\psi}^{-1}$, where $\boldsymbol{\psi}=\psi|\cdot|_{\mathbb{A}}^{-k}$ is athe central character of $\mathbf{f}$ (see Section 6 of \cite{Br11a}  for definition). Due to the criticality assumption $\lambda|_{\mathbb{A}^\times}=\boldsymbol{\psi}^{-1}$, the conductor of $\lambda$ is not independent of the conductor $c(\psi)$ of nebentypus $\psi$. Thus we provide the following

\begin{recipe} 
Note that regardless of the choice of $\lambda$ we have:
\begin{itemize}
\item at primes $l\nmid c(\psi)$ Hecke character $\lambda$ is unramified and
\item at non-split primes $l|c(\psi)$ the conductor of $\lambda_l$ divides $l^{\mathrm{ord}_l(c(\psi))}$ hence does not exceed $l^{\nu(l)}$.
\end{itemize}
However, if a split prime $l|c(\psi)$ we make a choice as follows. Note that $M_l^\times= M_{\bar{\mathfrak{l}}}^\times \times M_{\mathfrak{l}}^\times  =\mathbb{Q}_l^\times\times\mathbb{Q}_l^\times$ and consequently $\lambda_l = \lambda_{\bar{\mathfrak{l}}}\lambda_{\mathfrak{l}}$. Then
\begin{itemize}
\item at split primes $l|c(\psi)$ we choose $\lambda_l$ so that its conductor is supported at $\mathfrak{l}$, that is, we choose $\lambda_{\bar{\mathfrak{l}}}$  to be unramified and $\lambda_{\mathfrak{l}}$ to have conductor $\mathfrak{l}^{\mathrm{ord}_l(c(\psi))}$. 
\end{itemize}
\end{recipe}
Set
\[ f^{(p)}=\sum_{p\nmid j}a(j,f)q^j \; .\]
By adjoining the values of the arithmetic Hecke character $\lambda$ to $W$ we obtain a finite extension of $W$ which, abusing the symbol, we keep denoting $W$. Writing $\widehat{\lambda}:M^{\times} \backslash M^{\times}_{\mathbb{A}}\to W^\times$ for the $p$-adic avatar of $\lambda$ defined by $\widehat{\lambda}(x)=\lambda(x)x_p^k$,  the $W$-valued measures $\mathrm{d}\mu_j$ are given by
\[\mathrm{d}\mu_j=\widehat{\lambda}({\mathfrak{A}_j^{-1}})\mathrm{d}\mu_{f^{(p)},\mathfrak{A}_j} \qquad \text{for } j=1,\ldots,H^-.\]
Note that (\ref{key}) guarantees that the measures $\mathrm{d}\mu_{f^{(p)},\mathfrak{A}_j}$ are indeed supported on $\mathbb{Z}_p^\times$. Since $A_{j,p}=1$ we actually have $\widehat{\lambda}({\mathfrak{A}_j^{-1}})=\lambda({\mathfrak{A}_j^{-1}})$.

Let $\chi:M^{\times} \backslash M^{\times}_{\mathbb{A}}\to\mathbb{C}^\times$ be an anticyclotomic arithmetic Hecke character such that $\chi(a_\infty)=a_\infty^{m(1-c)}$ for some $m\geq 0$ and of conductor $N_{ns}p^n$, where $n\geq \mathrm{ord}_p(N)$ is an arbitrary integer, or of conductor $N_{ns}$. Write $\widehat{\chi}:M^{\times} \backslash M^{\times}_{\mathbb{A}}\to \mathbb{C}_p^\times$ defined by $\widehat{\chi}(x)=\chi(x)x_p^{m(1-c)}$ for its $p$-adic avatar. Denote by $I_M= M^\times \left\backslash M^{\times}_{\mathbb{A}}\right/ M^\times_\infty$ the idele class group. Since $(R\otimes_{\mathbb{Z}}\mathbb{Z}_p)^\times= R_{\bar{\mathfrak{p}}}^\times \oplus R_{\mathfrak{p}}^\times \cong\mathbb{Z}_p^\times\oplus\mathbb{Z}_p^\times$, the projection $\mathrm{pr}:I_M\rightarrow \mathrm{Cl}^-_\infty$ at $p$-th place is given by $\mathrm{pr}_p : x_p \mapsto x_p^{1-c}$. Thus, in the view of the coset decomposition (\ref{coset}), $\widehat{\chi}|_{\mathbb{Z}_p^\times}(z) =\tilde{\chi}_p(z)z^m$ for $z=x_p^{1-c}$ where 
$\tilde{\chi}_p=\chi_p\circ \mathrm{pr}_p$ is a primitive Dirchlet character of conductor $p^n$.

The Mazur--Mellin transform of measure $\mathrm{d}\mu_f$ given by
\[ \mathscr{L}(f;\xi)=\int_{\mathrm{Cl}^-_\infty}\xi{d}\mu_f\;, \qquad \xi \in \mathrm{Hom}_{cont}(\mathrm{Cl}^-_\infty,W^\times)\,,\]
is a $p$-adic analytic Iwasawa function on the $p$-adic Lie group $\mathrm{Hom}_{cont}(\mathrm{Cl}^-_\infty,W^\times)$ and, using (\ref{integration}), for the $p$-adic avatar $\widehat{\chi}$ we have: 
\[
\frac{\mathscr{L}(f;\widehat{\chi})}{\mathrm{\Omega}_p^{k+2m}}= \frac{G(\tilde{\chi}_p)}{\mathrm{\Omega}_p^{k+2m}} \sum_{j=1}^{H^-}\sum_{u\in (\mathbb{Z}/p^n\mathbb{Z})^\times}\lambda(\mathfrak{a}_{j,u}^{-1})\chi(\mathfrak{a}_{j,u}^{-1}) (d^mf^{(p)})(x(\mathfrak{a}_{j,u}),\omega_p(\mathfrak{a}_{j,u}))\,.
 \]

Here, for each $j=1,\ldots,H^-$, we have $x(\mathfrak{a}_{j,u})=x(\mathfrak{A}_j)\circ \bigl(\begin{smallmatrix} 1 & up^{-n} \\ 0 & 1 \end{smallmatrix} \bigr)$ on $Sh/\widehat{\Gamma}_1(N^{(p)}p^r)$, $u\in (\mathbb{Z}/p^n\mathbb{Z})^\times$, and $\mathfrak{a}_{j,u}$'s are exactly the $p^{n-1}(p-1)$ proper $R_{N_{ns}p^n}$-ideal class representatives such that $\mathfrak{a}_{j,u}R_{N_{ns}}=\mathfrak{A}_jR_{N_{ns}}$, that is, whose classes in $\mathrm{Cl}_n^-$ project down to the class of $\mathfrak{A}$ in $\mathrm{Cl}^-_{0}$. Note that the proper $R_{N_{ns}}p^n$-ideals $\mathfrak{a}_{j,u}$ in the above double sum form a complete set of class representatives for $\mathrm{Cl}_n^-$, and by means of the Hida's Waldspurger-type of formula from \cite{HidaNV} the double sum is related to the central critical value $L(\frac{1}{2},\hat{\pi}_{\mathbf{f}}\otimes (\lambda\chi)^-)$ ultimately yielding the interpolation formula
\begin{equation}\label{eq:interpolation}
\left(\frac{\mathscr{L}(f;\widehat{\chi})}{\mathrm{\Omega}_p^{k+2m}}\right)^2=C(\lambda,\chi,m)\frac{L(\frac{1}{2},\hat{\pi}_{\mathbf{f}}\otimes (\lambda\chi)^-)}{\left(\mathrm{\Omega}_\infty^{k+2m}\right)^2}
\end{equation}
 for an explicit constant $C(\lambda,\chi,m)$ (see Section 9 of \cite{Br11a} and the Main Theorem there in particular).

\section{Two variable anticyclotomic $p$-adic L-function attached to a Hida family}\label{sec:twovariables}

Write $\mathbf{p}$ for 4 or $p$ according to $p=2$ or not. Let $W_0$ be a (sufficiently large) discrete valuation ring finite flat over $\mathbb{Z}_p$. We choose and fix a generator $\gamma:=1 + \mathbf{p}$ of $\Gamma:=1+\mathbf{p}\mathbb{Z}_p$, which is the maximal torsion-free subgroup of $\mathbb{Z}_p^\times$, and identify the Iwasawa algebra $\Lambda := W_0[[\Gamma]]$ with the formal power series ring $W_0[[T]]$ via $\gamma \mapsto 1+T$. 

Let $\psi: \mathbb{Z}/Np^r\mathbf{p}\mathbb{Z}\to W_0^\times$ and consider the space of cusp forms $S_k(\Gamma_0(Np^r\mathbf{p}), \psi)$, $(p\nmid N, r\geq 0)$. Denote by $\mathbb{Z}[\psi]$ and $\mathbb{Z}_p[\psi]$ the rings generated by the values of $\psi$ over $\mathbb{Z}$ and $\mathbb{Z}_p$, respectively. The Hecke algebra over $\mathbb{Z}[\psi]$ is the subalgebra of the linear endomorphism algebra of $S_k(\Gamma_0(Np^r\mathbf{p}), \psi)$ generated by the Hecke operators 
\[h=Z[\psi][T(n)\mid n=1,2,\ldots]\subset \mathrm{End}(S_k(\Gamma_0(Np^r\mathbf{p}), \psi)\,.\]
If we set $h_{k,\psi}=h_{k,\psi/W_0}=h\otimes_{\mathbb{Z}[\psi]}W_0$, the ordinary Hecke algebra $\mathbf{h}^{\mathrm{ord}}_{k,\psi}\subset h_{k,\psi}$ is the maximal ring direct summand on which $U(p)$ is invertible. In other words, if we write $e=\lim_{n\to\infty}U(p)^{n!}$ under the $p$-adic topology of $h_{k,\psi}$, we have $\mathbf{h}^{\mathrm{ord}}_{k,\psi}= e h_{k,\psi}$. Let $\omega$ denote the Teichm\"uller character modulo $\mathbf{p}$. By utilizing the fixed embedding $\iota_p : \bar{\mathbb{Q}} \hookrightarrow \mathbb{C}_p$, the idempotent $e$ acts on the classical space $S_k(\Gamma_0(Np^r\mathbf{p}), \psi)$ as well, and we denote the image by $S^{\mathrm{ord}}_k(\Gamma_0(Np^r\mathbf{p}), \psi)$.

As constructed in \cite{Hi86a} and \cite{Hi86b}, the unique ``big'' ordinary Hecke algebra $\mathbf{h}^{\mathrm{ord}}$ is characterized by the following two properties usually referred to as Control theorems:
\begin{itemize}
\item[(C1)] $\mathbf{h}^{\mathrm{ord}}$ is free of finite rank over $\Lambda$,
\item[(C2)] When $k\geq 2$, $\mathbf{h}^{\mathrm{ord}}/(1+T-\gamma^k)\mathbf{h}^{\mathrm{ord}} \cong \mathbf{h}^{\mathrm{ord}}_{k,\psi_k} \text{ for }\psi_k:=\psi\omega^{-k}$.
\end{itemize}

We can let $a=(a_p,a_N)\in \mathbb{Z}_p^\times \times (\mathbb{Z}/N\mathbb{Z})$ act on $V(N;W_0)$ by
\[ f|\langle a\rangle (E,i_N,i_p) = f(E,a_Ni_N,a_p^{-1}i_p)\,.\]
For $s\in \mathbb{Z}_p$ and a character $\varepsilon: \Gamma \to \mu_{p^\infty}$, a $p$-adic modular form $f\in V(N;\tilde{W}_0)$ is said to have weight $(s,\varepsilon)$ if $f|\langle z\rangle = \varepsilon(z)z^sf$ for all $z\in \Gamma$.
A $\Lambda$-adic modular form is a formal $q$-expansion
\[\mathcal{F}(T;q)=\sum_{n\geq 0}a(n;\mathcal{F})(T)q^n\in \Lambda[[q]]\]
such that $\mathcal{F}(\gamma^s-1;q)$ is the $q$-expansion of a $p$-adic modular form of weight $(s,\cdot)$, for all $s\in \mathbb{Z}_p$. 
We denote the space of $\Lambda$-adic modular forms by $G(N;\Lambda)$. We call $\mathcal{F}$ a $\Lambda$-adic cusp form if $\mathcal{F}(\gamma^s-1;q)$ is a $p$-adic cusp form for all $s\in \mathbb{Z}_p$.
A $\Lambda$-adic cusp form is called arithmetic if for all sufficiently large positive integers $k$, $\mathcal{F}(\gamma^k-1;q)$ is the $q$-expansion of a classical cusp form of weight $k$. Denote by $S(N,\psi;\Lambda)$ (resp. $S^{\mathrm{ord}}(N,\psi;\Lambda)$) the space of all arithmetic $\Lambda$-adic cusp forms such that $\mathcal{F}(\gamma^k-1;q)\in S_k(\Gamma_0(Np^r\mathbf{p}), \psi_k;W_0)$ (resp. $\mathcal{F}(\gamma^k-1;q)\in S^{\mathrm{ord}}_k(\Gamma_0(Np^r\mathbf{p}), \psi_k;W_0)$ for all sufficiently large positive integers $k$. The space $S^{\mathrm{ord}}(N,\psi;\Lambda)$ is free of finite rank over $\Lambda$ and moreover is the $\Lambda$-dual of the ``big'' ordinary Hecke algebra $\mathbf{h}^{\mathrm{ord}}$. Thus, all the structural properties of $S^{\mathrm{ord}}(N,\psi;\Lambda)$ follow from the structural properties of $\mathbf{h}^{\mathrm{ord}}$, and in particular the specialization $\mathcal{F}\mapsto\mathcal{F}(\gamma^k-1;q)$ yields
\begin{itemize}
\item[(C$2^\prime$)] When $k\geq 2$, $S^{\mathrm{ord}}(N,\psi;\Lambda)\otimes_\Lambda \Lambda/(1+T-\gamma^k)\cong S^{\mathrm{ord}}_k(\Gamma_0(Np^r\mathbf{p}), \psi_k;W_0)$.
\end{itemize}
We refer to the family $\mathcal{F}=\{\mathcal{F}(\gamma^k-1;q)\}_k$ of ordinary  Hecke eigenforms as a $p$-adic analytic Hida family. Thus, any element in  $S^{\mathrm{ord}}_k(\Gamma_0(Np^r\mathbf{p}), \psi_k;W_0)$ for $k\geq 2$ can be lifted to a $p$-adic analytic Hida family.

Each element $\Phi(T)\in \Lambda$ gives rise to a bounded $p$-adic measure on $\Gamma$ by setting
\[\int_{\Gamma}z^s\mathrm{d}\Phi = \Phi (\gamma^s-1) \]
for all $p$-adic integers $s$. Without loss of generality we may assume $W_0 \subset W$. Note that the space $V(N;W)$ of $p$-adic modular forms over $W$ has a well defined norm $|f|_p=\mathrm{sup}_n|a(n,f)|_p$. For a given $\Lambda$-adic modular form $\mathcal{G}\in G(N;\Lambda)$, which is not necessarilly arithmetic and not necessarilly ordinary, following Section 3.2.4 of \cite{GME} we can define a bounded $p$-adic measure on $\Gamma$ with values  in $V(N;W)$ by
\begin{equation}\label{lambda-measure1}
\int_{\Gamma}z^s\mathrm{d}\mathcal{G}= \sum_{n\geq 0}a(n;\mathcal{G})(\gamma^s-1)q^n 
\end{equation}
for all $p$-adic integers $s$. 

Let $\mathcal{F}=\{\mathcal{F}(\gamma^k-1;q)\}_k=\{f_k\}_k\in S^{\mathrm{ord}}(N,\psi;\Lambda)$ be a $p$-adic analytic Hida family of Hecke eigenforms $f_k\in S^{\mathrm{ord}}_k(\Gamma_0(Np^r\mathbf{p}), \psi_k;W_0)$. We set
\[\mathcal{F}^{(p)}(T;q)=\sum_{(n,p)=1}a(n;\mathcal{F})(T)q^n\in \Lambda[[q]]\]
and note that $\mathcal{F}^{(p)}$, even though non-ordinary, is a $\Lambda$-adic modular form in $G(N;\Lambda)$ and hence it gives rise to a bounded $V(N;W)$-valued measure $\mathrm{d}\mathcal{F}^{(p)}$ on $\Gamma$ by means of (\ref{lambda-measure1}).

In a way consistent with the recipe provided in Section \ref{sec:acycpadicL}, we are going to associate to each $f_k$ an arithmetic Hecke character $\lambda_k$ of $M$ of $\infty$-type $(k,0)$ so that $\lambda_k|_{\mathbb{A}^\times}=\boldsymbol{\psi}_k^{-1}$ for the central character $\boldsymbol{\psi}_k:= \psi_k|\cdot|_{\mathbb{A}}^{-k}=\psi\omega^{-k}|\cdot|_{\mathbb{A}}^{-k}$ of $\pi_{\mathbf{f}_k}$. First, we fix once and for all a finite order arithmetic Hecke character $\tilde{\psi}:M^{\times} \backslash M^{\times}_{\mathbb{A}}\to\mathbb{C}^\times$ such  that $\tilde{\psi}|_{\mathbb{A}^\times}=\psi^{-1}$. Moreover, noting that for a split prime $l|c(\psi)$ we have $M_l^\times= M_{\bar{\mathfrak{l}}}^\times \times M_{\mathfrak{l}}^\times  =\mathbb{Q}_l^\times\times\mathbb{Q}_l^\times$ and consequently $\tilde{\psi}_l = \tilde{\psi}_{\bar{\mathfrak{l}}}\tilde{\psi}_{\mathfrak{l}}$,  we choose $\tilde{\psi}_l$ so that its conductor is supported at $\mathfrak{l}$, that is, we choose $\lambda_{\bar{\mathfrak{l}}}$  to be unramified and $\lambda_{\mathfrak{l}}$ to have conductor $\mathfrak{l}^{\mathrm{ord}_l(c(\psi))}$. Second, we fix once and for all an arithmetic Hecke character $\tilde{\omega}:M^{\times} \backslash M^{\times}_{\mathbb{A}}\to\mathbb{C}^\times$ of $\infty$-type $(1,0)$ such that $\tilde{\omega}|_{\mathbb{A}^\times}=\omega|\cdot|_{\mathbb{A}}$ and we choose $\tilde{\omega}_p=\tilde{\omega}_{\bar{\mathfrak{p}}}\tilde{\omega}_{\mathfrak{p}}$ so that its conductor is supported at $\mathfrak{p}$, that is, we choose $\tilde{\omega}_{\bar{\mathfrak{p}}}$  to be unramified and $\tilde{\omega}_{\mathfrak{p}}$ to have conductor $\mathfrak{p}$. Then for the Hecke eigenform $f_k$, the arithmetic Hecke character $\lambda_k:=\tilde{\psi}\tilde{\omega}^k$ is of $\infty$-type $(k,0)$, obeys the recipe provided in Section \ref{sec:acycpadicL}, and $\lambda_k|_{\mathbb{A}^\times}=\psi^{-1} \omega^k|\cdot|^k_{\mathbb{A}}= \boldsymbol{\psi}_k^{-1}$.

Recall that $\{\mathfrak{A}_1,\ldots,\mathfrak{A}_{H^-}\}$ is a complete set of representatives for $\mathrm{Cl}^-_0$ such that $\widehat{\mathfrak{A}}_j=A_j\widehat{R}_{N_{ns}}$ for $A_j\in M^{\times}_{\mathbb{A}}$, $j=1,\ldots,H^-$, chosen so that $A_{j,p}=1$ for all $j=1,\ldots,H^-$. Note that (\ref{coset}) yields a coset decomposition
\[ \mathrm{Cl}^-_\infty\times \Gamma =\bigsqcup_{j=1}^{H^-}\mathfrak{A}_j^{-1}\mathbb{Z}_p^\times\times \Gamma\,,\]
enabling us to construct a $W$-valued measure $\mathrm{d}\mu$ on $\mathrm{Cl}^-_\infty\times \Gamma$ by piecing together $H^-$ distinct $W$-valued measures $\mathrm{d}\mu_j$ on $\mathbb{Z}_p^\times\times \Gamma$, so that 
\begin{equation}\label{piece-together}
\int_{\mathrm{Cl}^-_\infty \times \Gamma}\phi_1\otimes\phi_2\mathrm{d}\mu:=\sum_{j=1}^{H^-}\phi_1(\mathfrak{A}^{-1}_j)\int_{\mathbb{Z}_p^\times\times \Gamma}\phi_1\otimes\phi_2\mathrm{d}\mu_j \text{ for } \phi_1\otimes\phi_2\in \mathrm{Cont}(\mathrm{Cl}^-_\infty\times \Gamma; W).
\end{equation}

The identification $\iota:\Gamma \cong \mathbb{Z}_p$ induces the indentification of measure spaces $\mathscr{M}(\Gamma;V(N;W))\cong \mathscr{M}(\mathbb{Z}_p;V(N;W))$, where a measure $\mu\in \mathscr{M}(\Gamma;V(N;W))$ corresponds to $\mu_+\in \mathscr{M}(\mathbb{Z}_p;V(N;W))$ via $\int_{\mathbb{Z}_p}\phi\mathrm{d}\mu_+=\int_{\Gamma}\phi\circ \iota\mathrm{d}\mu$. Furthermore, for any $\alpha\in \mathbb{Z}_p^\times$ we can define another measure $\mu_+|\alpha\in \mathscr{M}(\mathbb{Z}_p;V(N;W))$ by setting $\int_{\mathbb{Z}_p}\phi\mathrm{d}(\mu_+|\alpha)=\int_{\mathbb{Z}_p}\phi(\alpha x)\mathrm{d}\mu_+(x)$. In this way, for any $\mu\in \mathscr{M}(\Gamma;V(N;W))$ we can think of a measure $\mu|\alpha\in \mathscr{M}(\Gamma;V(N;W))$ which satisfies $\int_{\Gamma}z^k\mathrm{d}(\mu|\alpha)=\alpha^k\int_{\Gamma}z^k\mathrm{d}\mu$.
Combining this with (\ref{lambda-measure1}), for every $j=1,\ldots,H^-$, we define a bounded $V(N;W)$-valued $p$-adic measure $\mathrm{d}\mathcal{F}^{(p)}_{\mathfrak{A}_j}$ on $\Gamma$ by
\begin{equation}\label{eq:key}
\int_{\Gamma}z^k\mathrm{d}\mathcal{F}^{(p)}_{\mathfrak{A}_j}:= \widehat{\lambda}_k(\mathfrak{A}^{-1}_j)\mathcal{F}^{(p)}(\gamma^k-1;q)=\widehat{\tilde{\psi}\tilde{\omega}^k}(\mathfrak{A}^{-1}_j)f_k^{(p)}\,.
\end{equation}
Note that here $\widehat{\tilde{\psi}\tilde{\omega}^k}(\mathfrak{A}^{-1}_j)=\tilde{\psi}\tilde{\omega}^k(\mathfrak{A}^{-1}_j)$ as $A_{j,p}=1$.

For any $f\in V(N;W)$ one can define a $V(N;W)$-valued $p$-adic measure $\mathrm{d}f$ on $\mathbb{Z}_p$ by
\[\int_{\mathbb{Z}_p} \binom{x}{n}\mathrm{d}f=\binom{d}{n}f \quad \text{for all}\, n\geq 0\,. \]
Considering a CM point $x=(E_x,\eta^{(p)}_x, \eta^{\mathrm{ord}}_x\times \eta^{\mathrm{\acute{e}t}}_x)_{/\mathcal{W}}\in Sh(\mathcal{W})$ and an invariant differential $\omega_p$ on $E_{x/W}$, we can evaluate $\int_{\mathbb{Z}_p}\phi\mathrm{d}f\in V(N;W)$ at a pair $(x,\omega_p)_{/W}$, thus getting a bounded $p$-adic $W$-valued measure $\mathrm{d}f(x,\omega_p)$ on $\mathbb{Z}_p$. In other words, if we denote by $ev_{(x,\omega_p)}: V(N;W)\to W$ the evaluation homomorphism, we have that 
\begin{equation}\label{katz-measure2}
\int_{\mathbb{Z}_p}\phi\mathrm{d}f(x,\omega_p) := ev_{(x,\omega_p)}\left(\int_{\mathbb{Z}_p}\phi\mathrm{d}f\right) \text{ for } \phi \in \mathrm{Cont}(\Gamma;W)
\end{equation}
defines a bounded $p$-adic $W$-valued measure on ${\mathbb{Z}_p}$. Note that for the measure $\mu_{f,\mathfrak{A}_j}$ defined by (\ref{katzmeasure}) we have 
\[ \mathrm{d}f(x(\mathfrak{A}_j),\omega_p(\mathfrak{A}_j))=\mu_{f,\mathfrak{A}_j}\,,\]
and whenever the Fourier coefficients of $f$ are supported at integers prime to $p$ (i.e. $a(n,f)=0$ when $p\mid n$) the measure is supported on $\mathbb{Z}_p^\times$. Here the CM points $x(\mathfrak{A}_j)$ and the invariant differentials $\omega_p(\mathfrak{A}_j)$ are the ones constructed in Section \ref{sec:cmpts} and used throughout Section \ref{sec:acycpadicL}.

Thus, for $\phi_2\in \mathrm{Cont}(\Gamma,W^\times)$, by (\ref{eq:key}) we have $\int_\Gamma \phi_2\mathrm{d}\mathcal{F}^{(p)}_{\mathfrak{A}_j} \in V(N;W)$ and in the sense of (\ref{katz-measure2}) we can consider $\mathrm{d}\left( \int_\Gamma \phi_2\mathrm{d}\mathcal{F}^{(p)}_{\mathfrak{A}_j} \right)(x(\mathfrak{A}_j),\omega_p(\mathfrak{A}_j))$ as a bounded $p$-adic $W$-valued measure on ${\mathbb{Z}_p}$. Then for each $j=1,\ldots,H^-$, we define a $W$-valued measure $\mathrm{d}\mathcal{F}^j$ on $\mathbb{Z}_p^\times\times \Gamma$ by setting
\[\int_{\mathbb{Z}_p^\times\times \Gamma}\phi_1\otimes\phi_2\mathrm{d}\mathcal{F}^j:=\int_{\mathbb{Z}_p^\times}\phi_1\mathrm{d}\left(\int_\Gamma \phi_2\mathrm{d}\mathcal{F}^{(p)}_{\mathfrak{A}_j}\right)(x(\mathfrak{A}_j),\omega_p(\mathfrak{A}_j)) \text{ for } \phi_1\otimes\phi_2\in \mathrm{Cont}(\mathbb{Z}_p^\times\times \Gamma; W).\]
Using (\ref{piece-together}), we finally  get a $W$-valued measure $\mathrm{d}\mathcal{F}$ on $\mathrm{Cl}^-_\infty \times \Gamma$ by setting
\[\int_{\mathrm{Cl}^-_\infty \times \Gamma}\phi_1\otimes\phi_2\mathrm{d}\mathcal{F}:=\sum_{j=1}^{H^-}\phi_1(\mathfrak{A}^{-1}_j)\int_{\mathbb{Z}_p^\times\times \Gamma}\phi_1\otimes\phi_2\mathrm{d}\mathcal{F}^j \text{ for } \phi_1\otimes\phi_2\in \mathrm{Cont}(\mathrm{Cl}^-_\infty\times \Gamma; W).\]

As usual, we embedd $\mathbb{Z}_{\geq 2}$ into $\mathrm{Hom}_{cont}(\Gamma,W^\times)$ via $k\mapsto (\gamma \mapsto \gamma^k)$.
Note that by our construction the Mazur--Mellin transform $\mathcal{L}(\mathcal{F};\cdot,\cdot)$ of $\mathrm{d}\mathcal{F}$ satisfies
\[\mathcal{L}(\mathcal{F};\xi,k)=\mathscr{L}(f_k; \xi) \qquad \text{for } \xi \in \mathrm{Hom}_{cont}(\mathrm{Cl}^-_\infty,W^\times)\text{ and } k\geq 2,\]
where $\mathscr{L}(f_k;\cdot)$ is the Mazur--Mellin transform of the measure $\mu_{f_k}$ constructed in Section \ref{sec:acycpadicL}.
Thus, if $\chi:M^{\times} \backslash M^{\times}_{\mathbb{A}}\to\mathbb{C}^\times$ is an anticyclotomic arithmetic Hecke character such that $\chi(a_\infty)=a_\infty^{m(1-c)}$ for some $m\geq 0$ and of conductor $N_{ns}p^n$, where $n\geq r+ \mathrm{ord}_p(\mathbf{p})$ is an arbitrary integer, or of conductor $N_{ns}$ and $k\geq 2$, by (\ref{eq:interpolation}) we have
 
\begin{equation}\label{eq:familyinterpolation}
\left(\frac{\mathcal{L}(\mathcal{F};\widehat{\chi},k)}{\mathrm{\Omega}_p^{k+2m}}\right)^2=C(k,\tilde{\psi},\tilde{\omega},\chi,m)\frac{L(\frac{1}{2},\hat{\pi}_{\mathbf{f}_k}\otimes (\tilde{\psi}\tilde{\omega}^k\chi)^-)}{\left(\mathrm{\Omega}_\infty^{k+2m}\right)^2}
\end{equation}
for an explicit constant $C(k,\tilde{\psi},\tilde{\omega},\chi,m)$. This completes the proof of the Theorem \ref{main}.

\section{$\mathbb{I}$-adic version of the main result}\label{sec:i-adic}

We present the construction in the case of a ``full'' ordinary family of modular forms living on an irreducible component $\mathrm{Spec}(\mathbb{I})\subset\mathrm{Spec}(\mathbf{h}^{\mathrm{ord}})$,  such that  $\mathrm{Spec}(\mathbb{I})$ is a finite flat irreducible covering of $\mathrm{Spec}(\Lambda)$. Write $a(n)$ for the image of $T(n)$ in $\mathbb{I}$. Note that (C2) from previous section can be replaced by:

\begin{itemize}
\item[(C2($\mathbb{I})$)] When $k\geq 2$ and $\varepsilon: \Gamma \to \mu_{p^\infty}$ is a character, $\mathbf{h}^{\mathrm{ord}}/(1+T-\varepsilon(\gamma)\gamma^k)\mathbf{h}^{\mathrm{ord}} \cong \mathbf{h}^{\mathrm{ord}}_{k,\varepsilon\psi_k} \text{ for }\psi_k:=\psi\omega^{-k}$.
\end{itemize}

If a point $P\in \mathrm{Spec}(\mathbb{I})(\bar{\mathbb{Q}}_p)$ vanishes at $(1+T-\varepsilon(\gamma)\gamma^k)$ with $k\geq 2$, we call it an arithmetic point, and write $\varepsilon_P=\varepsilon$, $k_P=k$ and $p^{r(P)}$ for the order of $\varepsilon_P$. When $P=P_{(k,\varepsilon)}$ is arithmetic, by (C2($\mathbb{I})$)) there exists $f_{(k,\varepsilon)}\in S^{\mathrm{ord}}_k(\Gamma_0(Np^{r(P)}\mathbf{p}), \varepsilon\psi_k; \mathbb{I}/P_{(k,\varepsilon)})$ such that its $T(n)$-eigenvalue is given by $a_P(n):=P(a(n))\in \bar{\mathbb{Q}}_p$ for all $n$. Note that $\mathbb{I}/P_{(k,\varepsilon)}$ is a finite flat extension of $W_0$. Thus $\mathbb{I}$ gives rise to a family $\mathcal{F}=\{f_{(k,\varepsilon)}\mid \text{ arithmetic } P_{(k,\varepsilon)}\in \mathrm{Spec}(\mathbb{I}) \}$ of Hecke eigenforms which we call $p$-adic analytic Hida family. It is often referred to as ordinary or slope 0 family -- note that necessarily $|a_P(p)|_p=1$. Here the Hecke eigenvalue $a_P(n)$ for $T(n)$ is a $p$-adic analytic function on the rigid analytic space associated to the $p$-profinite formal spectrum $\mathrm{Spf}(\mathbb{I})$. Identifying $\mathrm{Spec}(\mathbb{I})(\bar{\mathbb{Q}}_p)$ with $\mathrm{Hom}_{W_0-\mathrm{alg}}(\mathbb{I},\bar{\mathbb{Q}}_p)$, each element $a\in \mathbb{I}$ gives rise to a function $a:\mathrm{Spec}(\mathbb{I})(\bar{\mathbb{Q}}_p)\to \bar{\mathbb{Q}}_p$ whose value at $(P:\mathbb{I}\to \bar{\mathbb{Q}}_p)$ is given by $a(P):=P(a)\in \bar{\mathbb{Q}}_p$. Let $\tilde{\mathbb{I}}$ denote the integral closure of $\Lambda$ in the quotient field of $\mathbb{I}$. Then $\mathrm{Spec}(\tilde{\mathbb{I}})$ is a finite covering of $\mathrm{Spec}(\mathbb{I})$ with surjection $\mathrm{Spec}(\tilde{\mathbb{I}})  \twoheadrightarrow \mathrm{Spec}(\mathbb{I})$, so abusing the defintion, we may regard the family $\mathcal{F}$ as being indexed by arithmetic points of $\mathrm{Spec}(\tilde{\mathbb{I}})(\bar{\mathbb{Q}}_p)$, where the arithmetic points $ \tilde{P}_{(k,\varepsilon)}\in \mathrm{Spec}(\tilde{\mathbb{I}})(\bar{\mathbb{Q}}_p)$ are the points above the arithmetic point $  P_{(k,\varepsilon)}\in \mathrm{Spec}(\mathbb{I})(\bar{\mathbb{Q}}_p)$. Let $\tilde{W}_0$ denote the integral closure of $W_0$ in $\bar{\mathbb{Q}}_p$.

In Section \ref{sec:acycpadicL} we chose a family of the central critical characters $\lambda_k$ obeying certain assumptions on its conductor described by the recipe in Section \ref{sec:acycpadicL}. In this section we construct the ``big'' character $\boldsymbol{\lambda} : M^{\times} \backslash M^{\times}_{\mathbb{A}}\to \mathbb{I}^\times$ such that for each $\tilde{P}_{(k,\varepsilon)}\in \mathrm{Spec}(\tilde{\mathbb{I}})(\bar{\mathbb{Q}}_p)$, the specializations $\lambda_{(k,\varepsilon)}:= \tilde{P}_{(k,\varepsilon)}(\boldsymbol{\lambda})$  restricted to the $\mathbb{A}^\times$ induces the inverse of the $p$-adic avatar of the central character of $\mathbf{f}_{(k,\varepsilon)}$.  Even though the specializations $\lambda_{(k,\varepsilon)}$  do not obey the requirements of the recipe provided in Section \ref{sec:acycpadicL}, if we assume Heegner hypothesis, namely that $N_{ns}=1$, the formula from a recent preprint \cite{Hs11} justifies the $p$-adic interpolation of the ``square root'' of the central critical L-value in this case.

Recall that for $s\in \mathbb{Z}_p$ and a character $\varepsilon: \Gamma \to \mu_{p^\infty}$, a $p$-adic modular form $f\in V(N;\tilde{W}_0)$ is said to have weight $(s,\varepsilon)$ if $f|\langle z\rangle = \varepsilon(z)z^sf$ for all $z\in \Gamma$.
For $\tilde{\mathbb{I}}$ as above, we define $\tilde{\mathbb{I}}$-adic modular form as a formal $q$-expansion
\[\mathcal{F}(\,\cdot\,;q)=\sum_{n\geq 0}a(n;\mathcal{F})(\,\cdot\,)q^n\in \tilde{\mathbb{I}}[[q]]\]
such that $\mathcal{F}(\tilde{P}_{(s,\varepsilon)};q)$ is the $q$-expansion of a $p$-adic modular form of weight $(s,\varepsilon$), for all weights $(s,\varepsilon)$. We denote the space of $\tilde{\mathbb{I}}$-adic modular forms by $G(N;\tilde{\mathbb{I}})$. We call $\mathcal{F}$ a $\tilde{\mathbb{I}}$-adic cusp form if $\mathcal{F}(\tilde{P}_{(s,\varepsilon)};q)$ is a $p$-adic cusp form for all weights $(s,\varepsilon)$.
A $\tilde{\mathbb{I}}$-adic cusp form is called arithmetic if for all sufficiently large positive integers $k$, and all characters  $\varepsilon: \Gamma \to \mu_{p^\infty}$, $\mathcal{F}(\tilde{P}_{(k,\varepsilon)};q)$ is the $q$-expansion of a classical cusp form of weight $k$. Denote by $S(N,\psi;\tilde{\mathbb{I}})$ (resp. $S^{\mathrm{ord}}(N,\psi;\tilde{\mathbb{I}})$) the space of all arithmetic $\tilde{\mathbb{I}}$-adic cusp forms such that $\mathcal{F}(\tilde{P}_{(k,\varepsilon)};q)\in S_k(\Gamma_0(Np^r\mathbf{p}), \varepsilon \psi_k;\tilde{\mathbb{I}}/\tilde{P}_{(k,\varepsilon)})$ (resp. $\mathcal{F}(\tilde{P}_{(k,\varepsilon)};q)\in S^{\mathrm{ord}}_k(\Gamma_0(Np^r\mathbf{p}), \varepsilon\psi_k;\tilde{\mathbb{I}}/\tilde{P}_{(k,\varepsilon)})$ for all sufficiently large positive integers $k$ and all characters  $\varepsilon: \Gamma \to \mu_{p^\infty}$. The specialization $\mathcal{F}\mapsto\mathcal{F}(\tilde{P}_{(k,\varepsilon)};q)$ yields
\begin{itemize}
\item[(C$2^{\prime\prime}$)] When $k\geq 2$, $S^{\mathrm{ord}}(N,\psi;\tilde{\mathbb{I}})\otimes_{\tilde{\mathbb{I}}} \tilde{\mathbb{I}}/\tilde{P}_{(k,\varepsilon)}\cong S^{\mathrm{ord}}_k(\Gamma_0(Np^r\mathbf{p}), \varepsilon \psi_k;\tilde{\mathbb{I}}/\tilde{P}_{(k,\varepsilon)})$.
\end{itemize}
We refer to the family $\mathcal{F}=\{\mathcal{F}(\tilde{P}_{(k,\varepsilon)};q)\}_{(k,\varepsilon)}$ of ordinary  Hecke eigenforms as a $p$-adic analytic Hida family. Thus any element in  $S^{\mathrm{ord}}_k(\Gamma_0(Np^r\mathbf{p}), \varepsilon\psi_k;\tilde{\mathbb{I}}/\tilde{P}_{(k,\varepsilon)})$ for $k\geq 2$ can be lifted to a $p$-adic analytic Hida family.

Let $\underline{\Lambda}=W_0[[\mathbb{Z}_p^\times]]$ and denote by $z\mapsto [z]$ the canonical inclusion of group-like elements $\mathbb{Z}_p^\times\hookrightarrow \underline{\Lambda}^\times$. Note that the space $V(N;\tilde{W}_0)$ of $p$-adic modular forms over $\tilde{W}_0$ has a well defined norm $|f|_p=\mathrm{sup}_n|a(n,f)|_p$.

\begin{i-adic} We consider $V(N;\tilde{\mathbb{I}})=V(N;\tilde{W}_0)\widehat{\otimes}_{\tilde{W}_0}\tilde{\mathbb{I}}$. Then
\[G(N;\tilde{\mathbb{I}})=\{f\in V(N;\tilde{\mathbb{I}})\mid f|\langle z \rangle = [z]f \text{ for all } z\in \Gamma\}\,. \]
\end{i-adic}
\begin{proof}
This is essentially proved as Theorem 3.2.16 on page 244 of \cite{GME} (the proof itself is on the page 243). The space $V(N,\tilde{\mathbb{I}})$ has two $\tilde{\mathbb{I}}$-module structures, one coming from the base ring $\tilde{\mathbb{I}}$ and another coming from the action of $\Gamma$ by means of the diamond operators $\langle z \rangle$. Each $\mathcal{F}\in V(N;\tilde{\mathbb{I}})$ such that $\mathcal{F}|\langle z \rangle = [z]\mathcal{F}$ for all $z\in \Gamma$, has a $q$-expansion at $\infty$ and, by definition we have a natural map
\[
V(N;\tilde{\mathbb{I}}) \otimes_{\tilde{\mathbb{I}}}\tilde{\mathbb{I}}/\tilde{P}_{(k,\varepsilon)}\to V(N;\tilde{\mathbb{I}}/\tilde{P}_{(k,\varepsilon)})
\]
for each $\tilde{P}_{(k,\varepsilon)}:\tilde{\mathbb{I}}\to \bar{\mathbb{Q}}_p$ taking $\mathcal{F}$ to $\mathcal{F}(\tilde{P}_{(k,\varepsilon)};q)$. The tensor product is taken using $\tilde{\mathbb{I}}$-module structure induced by the diamond operators. However, both $\tilde{\mathbb{I}}$-module structures coincide on $G(N;\tilde{\mathbb{I}})$ and consequently the map is injective by the $q$-expansion principle. Since the map takes $\mathcal{F}\in G(N;\tilde{\mathbb{I}})$ to a $p$-adic modular form of weight $(k,\varepsilon)$, we conclude that $\mathcal{F}\in G(N;\tilde{\mathbb{I}})$. 

Conversely, let $\mathcal{F}\in G(N;\tilde{\mathbb{I}})$ and we may regard it as a $V(N;\tilde{W}_0)$-valued bounded measure $\mathrm{d}\mathcal{F}$ on $\Gamma$ in the sense analogous to (\ref{lambda-measure1}). For each test object $(E,i_N,i_p)_{/R}$, and for a $p$-adic  $\tilde{\mathbb{I}}$-algebra $R$ which may be regarded as a $p$-adic $\tilde{W}_0$-algebra, we can evaluate $\int_\Gamma \phi \mathrm{d}\mathcal{F}$ at $(E,i_N,i_p)_{/R}$. In this way we obtain a bounded $\tilde{W}_0$-linear form from the space  $\mathrm{Cont}(\Gamma;\tilde{W}_0)$ into $R$, which we denote $\underline{\mathcal{F}}(E,i_N,i_p)\in R\widehat{\otimes}_{\tilde{W}_0} \tilde{\mathbb{I}}$.
Note that $R$ being a $\tilde{\mathbb{I}}$-algebra, together with the $\tilde{\mathbb{I}}$-module structure $\tilde{\mathbb{I}}\times R \to R$ given by $\alpha\otimes r=\alpha r$, induces a surjective algebra homomorphism $m_R: R\widehat{\otimes}_{\tilde{W}_0} \tilde{\mathbb{I}} \twoheadrightarrow R$. We then define $\mathcal{F}(E,i_N,i_p):=m_R(\underline{\mathcal{F}}(E,i_N,i_p))$. The assignment $(E,i_N,i_p)\mapsto \mathcal{F}(E,i_N,i_p)$ satisfies the axioms of a $p$-adic modular form defined over $\tilde{\mathbb{I}}$, i.e. gives rise to an element of $V(N;\tilde{\mathbb{I}})$. Moreover, this $p$-adic modular form satisfies $\mathcal{F}|\langle z \rangle = [z]\mathcal{F}$ for all $z\in \Gamma$, and has the same  $q$-expansion at $\infty$ as $\mathcal{F}$.
\end{proof}
Let $\mathrm{art}_\mathbb{Q}:\mathbb{A}^\times \to \mathrm{Gal}(\mathbb{Q}^\mathrm{ab}/\mathbb{Q})$ be the Artin reciprocity map.
\begin{critical-character} \label{critical-character}
There exists a central critical character $\boldsymbol{\lambda} : M^{\times} \backslash M^{\times}_{\mathbb{A}}\to \tilde{\mathbb{I}}^\times$ so that for each $\tilde{P}_{(k,\varepsilon)}\in \mathrm{Spec}(\tilde{\mathbb{I}})(\bar{\mathbb{Q}}_p)$  we have that $\lambda_{(k,\varepsilon)}:=\tilde{P}_{(k,\varepsilon)}(\boldsymbol{\lambda})$  restricted to $\mathbb{A}^\times$ induces the inverse of the $p$-adic avatar of the central character of $\mathbf{f}_{(k,\varepsilon)}$. 
\end{critical-character}
\begin{proof}
Following \cite{Ho}, a critical character is a homomorphism 
\[\theta : \mathbb{Z}_p^\times \to \tilde{\mathbb{I}}^\times\]
which satisfies $\theta^2(z)=[z]$ for all $z\in \mathbb{Z}_p^\times$. There are exactly two choices for a critical character, and they differ by multiplication by $\omega^\frac{p-1}{2}$. We fix one such choice and set
\[\Theta :G_\mathbb{Q} \to \tilde{\mathbb{I}}^\times\] 
by $\Theta(\sigma)=\theta(\epsilon_{\mathrm{cyc}}(\sigma))$, where $\epsilon_{\mathrm{cyc}}:G_\mathbb{Q} \to \mathbb{Z}_p^\times$ is the cyclotomic character. Then the homomorphism $[\cdot]:\mathbb{Z}_p^\times\to \tilde{\mathbb{I}}^\times$ composed with $\tilde{P}_{(k,\varepsilon)}\in \mathrm{Spec}(\tilde{\mathbb{I}})(\bar{\mathbb{Q}}_p)$ is given by
\[ [z]_{(k,\varepsilon)}= \varepsilon(z)\psi(z)\omega^{-k}(z)z^k\]
for all $z\in \mathbb{Z}_p^\times$. Note that $[z]_{(k,\varepsilon)}$ is the $p$-adic avatar of the central character $\varepsilon\psi\omega^{-k}|\cdot|^{-k}_{\mathbb{A}}$ of $\mathbf{f}_{(k,\varepsilon)}$ because the $p$-adic avatar of the norm character induces the cyclotomic character.
We denote by $\theta_{(k,\varepsilon)}$ and $\Theta_{(k,\varepsilon)}$ the compositions
\[\theta_{(k,\varepsilon)}:\mathbb{Z}_p^\times\stackrel{\theta}{\longrightarrow}\tilde{\mathbb{I}}^\times \stackrel{\tilde{P}_{(k,\varepsilon)}}{\longrightarrow} \bar{\mathbb{Q}}_p \qquad 
\Theta_{(k,\varepsilon)}: G_\mathbb{Q} \stackrel{\Theta}{\longrightarrow}\tilde{\mathbb{I}}^\times \stackrel{\tilde{P}_{(k,\varepsilon)}}{\longrightarrow} \bar{\mathbb{Q}}_p\,. \]
We define a $\bar{\mathbb{Q}}_p$-valued character of $M^{\times}_{\mathbb{A}}$
\[\lambda_{(k,\varepsilon)}(x) = \Theta_{(k,\varepsilon)}(\mathrm{art}_\mathbb{Q}(\mathrm{Norm}_{M/\mathbb{Q}}(x)))\,,\]
and $\boldsymbol{\lambda} : M^{\times} \backslash M^{\times}_{\mathbb{A}}\to \mathbb{I}^\times$ by 
\[\boldsymbol{\lambda}(x)=\Theta(\mathrm{art}_\mathbb{Q}(\mathrm{Norm}_{M/\mathbb{Q}}(x)))\,.\]
Note that
\[ \mathbb{Z}_p^\times \stackrel{x\mapsto x_p}{\longrightarrow} \mathbb{A}^\times \stackrel{\mathrm{art}_\mathbb{Q}}{\longrightarrow} \mathrm{Gal}(\mathbb{Q}(\mu_{p^\infty})/\mathbb{Q})\stackrel{\epsilon_{\mathrm{cyc}}}{\longrightarrow} \mathbb{Z}_p^\times \]  
is given by $x\mapsto x^{-1}$. The nebentypus of $f_{(k,\varepsilon)}$ is $\varepsilon\psi\omega^{-k}$, and as we noted above, the $p$-adic avatar of the central character of $\mathbf{f}_{(k,\varepsilon)}$ is $[z]_{(k,\varepsilon)}=\theta_{(k,\varepsilon)}^2$. It follows that
\[ \lambda_{(k,\varepsilon)}|_{\mathbb{A}^\times}(x)= \lambda_{(k,\varepsilon)}(x_p)=\Theta^2_{(k,\varepsilon)}(\mathrm{art}_\mathbb{Q}(x_p))= \theta_{(k,\varepsilon)}^{-2}\]
as desired, where $x\in\mathbb{Z}_p^\times$, $x\in{\mathbb{A}^\times}$ is the idele with $x$ in the $p$-component and 1 in all other components.
\end{proof}
Take an extension of the ring of Witt vectors $W$ obtained by adjoining $\tilde{W}_0$ and, by abuse of notation, keep denoting it $W$.
Let $x(\mathfrak{A})_{/\tilde{\mathbb{I}}}=x(\mathfrak{A})_{/W}\otimes_W\tilde{\mathbb{I}}$.
We define a $\tilde{\mathbb{I}}$-valued measure $\mathrm{d}\mu_{\mathcal{F},{\mathfrak{A}}}$ on $\mathbb{Z}_p$
\[
\int_{\mathbb{Z}_p} \binom{x}{n}\mathrm{d}\mu_{\mathcal{F},{\mathfrak{A}}}=\binom{d}{n}\mathcal{F}(x(\mathfrak{A}),\omega_p(\mathfrak{A})) \quad \text{for all}\, n\geq 0\,. 
\]
Setting
\[d\mu_j=\boldsymbol{\lambda}(\mathfrak{A}_j^{-1})\mathrm{d}\mu_{\mathcal{F}^{(p)},\mathfrak{A}_j}\]  
and piecing these measures $\mu_j$ together, we get a $\tilde{\mathbb{I}}$-valued measure $\mathrm{d}\mu_\mathcal{F}$ on $\mathrm{Cl}^-_\infty$ so that for every continuous function $\xi:\mathrm{Cl}^-_\infty\to \tilde{\mathbb{I}}$ we have 
\[\int_{\mathrm{Cl}^-_\infty} \xi\mathrm{d}\mu_{\mathcal{F}}=\sum_{j=1}^{H^-}\int_{\mathbb{Z}_p^\times}\xi(\mathfrak{A}_j^{-1}x)\mathrm{d}\mu_j(x)\,.\]
Note that for an arithmetic point $\tilde{P}_{(k,\varepsilon)}\in \mathrm{Spec}(\tilde{\mathbb{I}})(\bar{\mathbb{Q}}_p)$ 
with $k\geq 2$, we have
\[\tilde{P}_{(k,\varepsilon)}(\mu_{\mathcal{F}})=\mu_{f_{(k,\varepsilon)}}\,,\]
where $\mu_{f_{(k,\varepsilon)}}$ is the measure attached to the cusp form $f_{(k,\varepsilon)}$ in Section \ref{sec:acycpadicL}.
We can define a $p$-adic analytic function $\mathcal{L}(\mathcal{F;\cdot,\cdot})$ on $\mathrm{Spf}(\tilde{\mathbb{I}})\times \mathrm{Cl}^-_\infty$ by
\[ \mathcal{L}(\mathcal{F};\tilde{P},\xi)=\int_{\mathrm{Cl}^-_\infty} \xi \mathrm{d}\tilde{P}(\mu_{\mathcal{F}})\,.\]
As a matter of fact, this is an Iwasawa function on $\mathrm{Spec}(W[[ \mathrm{Cl}^-_\infty]]\times \tilde{\mathbb{I}})$ in the sense that it is a global section of the structure sheaf of  $\mathrm{Spec}(W[[\mathrm{Cl}^-_\infty]]\times \tilde{\mathbb{I}})$,  i.e. an element in $W[[\mathrm{Cl}^-_\infty]]\times \tilde{\mathbb{I}}$. 

If the assume the Heegner hypothesis, namely that $N_{ns}=1$, the $p$-adic interpolation of the central critical L-value is justified by the Waldspurger-type of formula proved in a recent preprint \cite{Hs11} (Theorem A). Note that the Hypothesis A of that theorem is satisfied once we assume the Heegner hypothesis. The formula in Theorem A of \cite{Hs11} is different from the Hida's Walsdspurger-type of formula from \cite{HidaNV}. The heart of Hida's Waldspurger-type of formula is a delicate choice of a Schwartz--Bruhat function on $M_2(\mathbb{A})$ attaining the optimality of theta correspondence (Sections 1.4 and 1.7 of \cite{HidaNV}), and our construction of CM points (see Section 5.1 of \cite{Br11a}) is from the $p$-adic point of view precisely tailored for this choice. On the other hand, \cite{Hs11} follows ostensibly different approach of choosing local Whittaker functions in the process of decomposing the global period toric integrals into products of local ones, ultimately resulting in different hypotheses from the work of Hida. 

Thus, if $\chi:M^{\times} \backslash M^{\times}_{\mathbb{A}}\to\mathbb{C}^\times$ is an anticyclotomic arithmetic Hecke character such that $\chi(a_\infty)=a_\infty^{m(1-c)}$ for some $m\geq 0$ and of $p$-power conductor, and $k\geq 2$, we have
 
\begin{equation}\label{eq:i-adic-familyinterpolation}
\left(\frac{\mathcal{L}(\mathcal{F};\tilde{P}_{(k,\varepsilon)},\widehat{\chi})}{\mathrm{\Omega}_p^{k+2m}}\right)^2=C(k,\varepsilon,\lambda)\frac{L(\frac{1}{2},\hat{\pi}_{\mathbf{f}_{(k,\varepsilon)}}\otimes \tilde{P}_{(k,\varepsilon)}(\boldsymbol{\lambda})\chi)}{\left(\mathrm{\Omega}_\infty^{k+2m}\right)^2}
\end{equation}
for an explicit constant $C(k,\varepsilon,\lambda)$.

\end{document}